\documentclass[11pt, a4paper, reqno]{amsart}

\usepackage{amsmath, amssymb, amsthm, mathtools}
\usepackage{tikz-cd}
\usepackage{enumitem}
\usepackage{graphicx}
\usepackage{xcolor}
\usepackage[colorlinks=true, linkcolor=blue, citecolor=red, urlcolor=blue]{hyperref}
\usepackage[capitalize,nameinlink]{cleveref}
\usepackage{float}
\usepackage{caption}

\newtheorem{theorem}{Theorem}[section]
\newtheorem{lemma}[theorem]{Lemma}
\newtheorem{proposition}[theorem]{Proposition}

\theoremstyle{definition}
\newtheorem{definition}[theorem]{Definition}
\newtheorem{example}[theorem]{Example}

\theoremstyle{remark}
\newtheorem{remark}[theorem]{Remark}
\newtheorem{claim}[theorem]{Claim}

\DeclareMathOperator{\Aut}{Aut}
\DeclareMathOperator{\Pic}{Pic}
\DeclareMathOperator{\Cl}{Cl}
\DeclareMathOperator{\Spec}{Spec}
\DeclareMathOperator{\Proj}{Proj}

\newcommand{\R}{\mathbb{R}}
\newcommand{\C}{\mathbb{C}}
\newcommand{\Z}{\mathbb{Z}}
\newcommand{\Q}{\mathbb{Q}}
\newcommand{\PP}{\mathbb{P}}

\newcommand{\GG}{\mathbb{G}}

\newcommand{\Autzero}{\Aut^0}

\title[Symmetries of Toric Varieties over Characteristic Zero]{On Maximal Symmetries of Toric Varieties over Fields of Characteristic Zero}

\author{Yutaro Naito}
\address{Graduate School of Mathematics, Nagoya University, Furo-cho, Chikusa-ku, Nagoya, 464-8602, Japan}
\email{yutaro.naito.h5@math.nagoya-u.ac.jp}


\begin{document}

\begin{abstract}
In this paper, we study complete simplicial toric varieties admitting faithful actions of large symmetric groups. First, we correct a recent classification result by Esser, Ji, and Moraga concerning $4$-dimensional toric varieties with $S_6$-actions over the complex numbers $\mathbb{C}$, providing the complete list of such varieties. Second, we extend the study of maximal symmetric group actions to non-closed fields $k$ of characteristic zero satisfying a certain arithmetic condition (such as $\mathbb{Q}$ or $\mathbb{R}$). Over such fields, we reveal a striking rigidity in dimensions $n \neq 2$, where the maximal symmetric action uniquely restricts the variety to the projective space $\mathbb{P}^n_k$. In sharp contrast, for dimension $n=2$, we discover and classify an infinite family of split and non-split toric surfaces admitting faithful $S_4$-actions by utilizing the equivariant Minimal Model Program and Galois descent.
\end{abstract}

\maketitle

\tableofcontents

\section{Introduction}
\label{sec:intro}

The study of finite group actions on algebraic varieties is a classical and central topic in algebraic geometry. Among various algebraic varieties, toric varieties provide a rich testing ground because their geometric properties can be completely translated into the combinatorial data of fans and lattices. Recently, Esser, Ji, and Moraga \cite{esser2025symmetries} systematically studied finite group actions on complete simplicial toric varieties. They determined the maximum degree $m$ of the symmetric group $S_m$ that can act faithfully on an $n$-dimensional complete simplicial toric variety over the complex numbers $\mathbb{C}$, and they classified the varieties realizing these maximal actions for many dimensions.

However, their classification over $\mathbb{C}$ contains a gap in dimension $n=4$. Furthermore, their study heavily relies on the assumption that the base field is algebraically closed. When we consider toric varieties over non-closed fields, such as the field of rational numbers $\mathbb{Q}$ or the field of real numbers $\mathbb{R}$, the representation theory of finite groups imposes strong arithmetic constraints. Some symmetries that are geometrically possible over $\mathbb{C}$ cannot be realized over non-closed fields. This motivates us to correct the classification over $\mathbb{C}$ and to explore the entirely new landscape of maximal symmetric actions over general fields.

In this paper, we focus on a base field $k$ of characteristic zero satisfying a specific arithmetic condition (\hyperref[cond:star]{\rm ($\star$)} in \cref{sec:non_closed}), which requires that the equation $a^2 + b^2 = -1$ has no solutions in $k$. Under this natural condition, we completely classify the $n$-dimensional complete simplicial toric varieties over $k$ admitting the maximal symmetric group actions. We show that while the geometry is extremely rigid for dimensions $n \neq 2$, the $2$-dimensional case exhibits a surprisingly rich structure, yielding an infinite family of non-split toric surfaces.

The paper is organized as follows:

In \cref{sec:preliminaries}, we recall the basic notions of toric varieties, Cox rings, and their automorphism groups. We also review the definitions of split and non-split toric varieties over non-closed fields, the equivariant Minimal Model Program for surfaces, and the Weil restriction of scalars, which will be essential tools for our classification.

In \cref{sec:dimension_4}, we restrict our attention to the base field $\mathbb{C}$. We point out an error in the classification of $4$-dimensional toric varieties with $S_6$-actions given by Esser, Ji, and Moraga \cite[Theorem 4.1]{esser2025symmetries} by providing a counterexample. We then provide the correct and complete classification by proving the following theorem:

\vspace{0.5em}
\noindent \textbf{Theorem A (\cref{thm:main2}).} \textit{Let $X$ be a complete simplicial toric variety of dimension $4$ over $\mathbb{C}$. Then $S_6 \le \Autzero(X)$ if and only if $X$ is isomorphic to one of the following:
\begin{itemize}
    \item the projective space $\mathbb{P}^4$,
    \item the projective bundle $\mathbb{P}(\mathcal{O}(a) \oplus \mathcal{O}) \to \mathbb{P}^3$ for some even integer $a$, or
    \item the weighted projective space $\mathbb{P}(1,1,1,1,m)$ for some positive even integer $m$.
\end{itemize}}
\vspace{0.5em}

In \cref{sec:non_closed}, we present our main results over non-closed fields. In this setting, we obtain results in lower dimensions that differ from those over $\mathbb{C}$. Under the condition \hyperref[cond:star]{($\star$)}, we determine the maximum degree of symmetric groups acting on $n$-dimensional toric varieties over $k$ and explicitly classify the corresponding varieties:

\vspace{0.5em}
\noindent \textbf{Theorem B (\cref{thm:main_nonclosed}).} \textit{Let $k$ be a field satisfying the condition \hyperref[cond:star]{\rm ($\star$)}. Let $m$ be the maximum degree of the symmetric group $S_m$ acting faithfully on an $n$-dimensional complete simplicial toric variety $X$ over $k$. Then we always have $m = n+2$, and the isomorphism classes of $X$ are given as follows:
\begin{itemize}
    \item If $n \neq 2$, then $X \cong \mathbb{P}^n_k$.
    \item If $n = 2$, then $X$ belongs to an infinite family of split and non-split toric surfaces (see \cref{prop:dim2_classification}).
\end{itemize}}
\vspace{0.5em}

As explicitly stated in Theorem B, we prove in \cref{subsec:dim1_and_5} and \cref{subsec:dim3} that the geometry is extremely rigid for dimensions $n \neq 2$, uniquely determining the variety to be the projective space $\mathbb{P}^n_k$. Finally, in \cref{subsec:dim2}, we investigate the exceptional case of dimension $n=2$. By running the $S_4$-equivariant Minimal Model Program and utilizing Galois descent, we successfully construct the aforementioned infinite family of split and non-split toric surfaces admitting faithful $S_4$-actions and completely characterize their geometric structures.

\addtocontents{toc}{\protect\setcounter{tocdepth}{0}}

\subsection*{Acknowledgments}
The author would like to express his sincere gratitude to his advisor, Professor Sho Tanimoto, for his continuous guidance, valuable discussions, and encouragement throughout this work. The author is also deeply grateful to L.~Esser, L.~Ji, and J.~Moraga for helpful correspondences and fruitful discussions concerning the results in \cref{sec:dimension_4}. 
This work was supported by the Make New Standard Program for the Next Generation Researchers of the Tokai National Higher Education and Research System (THERS).

\addtocontents{toc}{\protect\setcounter{tocdepth}{2}}

\section{Preliminaries}
\label{sec:preliminaries}

\subsection{Toric Varieties and Cox Rings over \texorpdfstring{$\C$}{C}}

Let $N \cong \Z^n$ be a lattice of rank $n$, and let $M = \operatorname{Hom}(N, \Z)$ be its dual lattice. Let $T = \operatorname{Spec} \C[M]$ be the algebraic torus over $\C$ associated with $N$. 
A toric variety $X$ over $\C$ is defined by a fan $\Delta$ in $N_\R = N \otimes_\Z \R$. We assume that $X$ is complete and simplicial. 

Let $\Delta(1)$ be the set of $1$-dimensional cones (rays) in $\Delta$. For each ray $\rho \in \Delta(1)$, we denote by $v_\rho \in N$ its primitive generator, and by $D_\rho$ the corresponding torus-invariant prime divisor on $X$. 
Since $X$ is simplicial, the divisor class group $\Cl(X)$ is a finitely generated abelian group. The Cox ring of $X$ is defined as the polynomial ring
$$R(X) = \C[x_\rho \mid \rho \in \Delta(1)],$$
which is naturally graded by $\Cl(X)$ via $\deg(x_\rho) = [D_\rho]$. The toric variety $X$ can be recovered as a geometric quotient of an open subset of $\Spec R(X)$ by the action of the characteristic quasitorus $G = \operatorname{Hom}(\Cl(X), \mathbb{G}_m)$.

\subsection{Automorphism Groups of Toric Varieties over \texorpdfstring{$\C$}{C}}

Let $\Aut(X)$ be the automorphism group of $X$ as an algebraic variety over $\C$, and let $\Aut^0(X)$ be its identity component. 
The structure of $\Aut(X)$ is deeply related to the symmetries of the fan $\Delta$ and the gradings of the Cox ring $R(X)$.

Let $\Aut(N, \Delta)$ be the subgroup of $\operatorname{GL}(N)$ consisting of automorphisms that preserve the fan $\Delta$. The set of rays $\Delta(1)$ is partitioned into linear equivalence classes based on the classes $[D_\rho]$ in $\Cl(X)$. Suppose there are $s$ such classes, and let $d_i$ be the number of rays in the $i$-th class for $1 \le i \le s$. 
Let $S_{\Delta_i} \cong S_{d_i}$ be the symmetric group acting by permuting the rays within the $i$-th linear equivalence class. 

\begin{theorem}[{\cite[Theorem 4.3]{esser2025symmetries}, \cite[\S4]{MR1299003}}] \label{thm:aut_structure}
    Let $X = X(\Delta)$ be a complete simplicial toric variety over $\C$. We denote by $\Aut_g(R(X))$ the group of automorphisms of $R(X)$ preserving the grading. 
    \begin{enumerate}[label={\rm (\arabic*)}]
        \item $\Aut_g(R(X))$ is isomorphic to the semidirect product $U \rtimes G_s$, where $U$ is the unipotent radical and $G_s = \prod_{i=1}^{s}\operatorname{GL}_{d_i}(\C)$. In particular, any finite subgroup of $\Aut_g(R(X))$ is conjugate to a subgroup of $G_s$.
        \item $\Aut^0(X) \cong \Aut_g(R(X))/G$.
        \item $\Aut(N,\Delta) \hookrightarrow S_{\Delta(1)}$ and
            \[
                \Aut(X)/\Aut^0(X) \cong \Aut(N,\Delta)/\prod_{i=1}^s S_{d_i}.
            \]
    \end{enumerate}
\end{theorem}

\subsection{Toric varieties over an arbitrary field}

\begin{definition}
Let $k$ be a field. An algebraic group $T$ over $k$ is called an \textit{algebraic torus} if $T_{\overline{k}} \cong \mathbb{G}_{m, \overline{k}}^n$ after base change to an algebraic closure $\overline{k}$, where $n$ is the dimension of the torus. 
A \textit{toric variety} $X$ over $k$ is a normal variety over $k$ containing an algebraic torus $T$ as an open dense subset, such that the action of $T$ on itself extends to an algebraic action $T \times X \to X$.
\end{definition}

\begin{remark}
When the base field is algebraically closed (i.e., $k = \overline{k}$), we always have $T \cong \mathbb{G}_m^n$. As is well known, a toric variety $X$ over an algebraically closed field is completely determined by a fan $\Delta$ in the cocharacter lattice $N = \operatorname{Hom}(\mathbb{G}_m, T) \cong \mathbb{Z}^n$.
\end{remark}

\begin{definition}
If $T \cong \mathbb{G}_{m,k}^n$ over $k$, $T$ is called a \textit{split torus}. A toric variety $X$ equipped with an action of a split torus is called a \textit{split toric variety} if every irreducible component of the torus-invariant divisor on $X$ is geometrically integral over $k$. 
For a general algebraic torus $T$, there exists a finite Galois extension $L/k$ such that $T_L := T \times_k \operatorname{Spec} L$ is a split torus over $L$. In this case, the Galois group $\operatorname{Gal}(L/k)$ acts on the cocharacter lattice $N$ (see, for example, \cite[Chapter 3, \S8.2]{Vos98}). Such a field $L$ is called a \textit{splitting field} for $T$.
\end{definition}

\begin{example}
Consider the circle group $S^1 = \operatorname{Spec} \mathbb{R}[x,y]/(x^2+y^2-1)$ over the field of real numbers $\mathbb{R}$. Upon base change to the splitting field $\mathbb{C}$, letting $u = x+iy$ and $v = x-iy$ yields $uv=1$, which implies $S^1_{\mathbb{C}} \cong \mathbb{G}_{m, \mathbb{C}}$. Thus, $S^1$ is a torus over $\mathbb{R}$. 
However, considering the topological properties of the $\mathbb{R}$-rational points, $S^1(\mathbb{R})$ is compact, whereas $\mathbb{G}_m(\mathbb{R}) \cong \mathbb{R}^\times$ is not compact. Therefore, $S^1 \not\cong \mathbb{G}_{m,\mathbb{R}}$, which means $S^1$ is a non-split torus over $\mathbb{R}$. In this case, the action of $\operatorname{Gal}(\mathbb{C}/\mathbb{R}) \cong \mathbb{Z}/2\mathbb{Z}$ on the lattice $N \cong \mathbb{Z}$ is given by the inversion $v \mapsto -v$.
\end{example}

From the above facts, a toric variety $X$ over $k$ is completely characterized by the following pair of data:
\begin{enumerate}
    \item A split toric variety $X_L$ over a splitting field $L$ (which is equivalent to a fan $\Delta$ in the lattice $N$).
    \item A faithful action of the Galois group $\operatorname{Gal}(L/k)$ on the lattice $N$ that leaves the fan $\Delta$ invariant.
\end{enumerate}

\subsection{Equivariant Minimal Model Program for Toric Surfaces}
\label{subsec:equivariant_mmp}

In this subsection, we focus on complete simplicial split toric surfaces (i.e., dimension $n=2$). Let $G$ be a finite group acting on a split toric surface $X$ via algebraic automorphisms. 
We study the birational geometry of $X$ through the $G$-equivariant Minimal Model Program ($G$-MMP). The general framework of MMP can be found in \cite{kollar1998birational}, and the equivariant surface case is thoroughly detailed in \cite{dolgachev2009finite}.

Recall from \cite[Theorem 4.1.3]{cox2011toric} the standard short exact sequence of real vector spaces associated to $X$:
$$ 0 \to M_\R \to \R^{\Delta(1)} \to \Pic(X)_\R \to 0, $$
where $\R^{\Delta(1)}$ is generated by the torus-invariant prime divisors. Since $G$ acts on $X$, it naturally acts on this sequence. Taking the $G$-invariant part, which is an exact functor because we are working over $\R$ and $G$ is finite, we obtain the exact sequence:
$$ 0 \to (M_\R)^G \to (\R^{\Delta(1)})^G \to \Pic(X)_\R^G \to 0. $$
If the $G$-action on the lattice $N$ fixes no non-trivial vectors, we have $N^G = \{0\}$, which implies its dual $(M_\R)^G = \{0\}$. Consequently, the $G$-invariant Picard number $\rho^G(X) := \dim_\R \Pic(X)_\R^G$ is exactly equal to $\dim_\R (\R^{\Delta(1)})^G$. Since $\R^{\Delta(1)}$ is a permutation representation of $G$ on the set of rays $\Delta(1)$, the dimension of its invariant subspace exactly equals the number of $G$-orbits in $\Delta(1)$.

During the $G$-MMP, any $G$-equivariant extremal contraction $\phi \colon X \to Y$ preserves the toric structure. This is because the algebraic torus $T$ and the finite group $G$ generate a group isomorphic to $T \rtimes G$ acting on $X$, and the contraction of a $T$-invariant (and $G$-invariant) extremal face naturally induces a toric structure on the target variety $Y$ (see e.g., \cite[Chapter 15]{cox2011toric} or \cite[Section 14-1]{matsuki2002introduction}).

In the specific case of toric surfaces, a step of the $G$-MMP consists of the simultaneous $G$-equivariant contraction of a $G$-orbit of $K_X$-negative extremal curves. Since $X$ is a toric surface, the contraction of a curve is a divisorial contraction. According to the general theory of the Toric MMP (see \cite[Section 15.4]{cox2011toric} and \cite[Theorem 14-1-5]{matsuki2002introduction}), a toric divisorial contraction of a torus-invariant prime divisor $D_\rho$ corresponds precisely to removing its associated ray $\rho$ from the fan $\Delta$ and replacing the adjacent maximal cones with their union. 

Therefore, a divisorial step of the $G$-MMP on $X$ translates purely combinatorially into removing a $G$-orbit of rays from $\Delta(1)$. Running the $G$-MMP terminates with a $G$-Mori fiber space, where $\rho^G$ is $1$ or $2$. In \cref{sec:non_closed}, this combinatorial constraint will strictly bound the possible fan structures of $X$.

\subsection{Weil Restriction of Scalars}
\label{subsec:weil_res}
In this subsection, we recall the notion of Weil restriction (or restriction of scalars), which will play a crucial role in constructing a non-split $k$-form with an $S_5$-action in \cref{sec:non_closed}. For a comprehensive and scheme-theoretic treatment of this topic, we refer the reader to \cite[\S7.6]{BLR90}.

\begin{definition}
  Let $L/k$ be a finite field extension. The Weil restriction functor $\operatorname{Res}_{L/k}$ from the category of $L$-schemes to the category of $k$-schemes is defined as the right adjoint to the base change functor $-\times_{\operatorname{Spec} k} \operatorname{Spec} L$. 
\end{definition}

\begin{example}
To observe the explicit geometric structure, let us consider the specific fields $k = \mathbb{Q}(\sqrt{-7})$ and its quadratic extension $L = k(\sqrt{5})$. We apply the Weil restriction to the projective line $\mathbb{P}^1_L$ to obtain a $k$-variety $X = \operatorname{Res}_{L/k}(\mathbb{P}^1_L)$. Since $[L:k] = 2$, $X$ is a surface over $k$.

Let $[Z_0 : Z_1]$ be the homogeneous coordinates of $\mathbb{P}^1_L$. Let $\sigma \in \operatorname{Gal}(L/k)$ be the non-trivial Galois automorphism, which sends $\sqrt{5}$ to $-\sqrt{5}$. We can construct $k$-rational coordinates $U_0, U_1, U_2, U_3 \in k$ by defining:
\[
U_0 = Z_0 \sigma(Z_0), \quad U_1 = Z_1 \sigma(Z_1), \quad \text{and} \quad Z_0 \sigma(Z_1) = U_2 + \sqrt{5} U_3.
\]
Using the norm relation $(Z_0 \sigma(Z_0))(Z_1 \sigma(Z_1)) = (Z_0 \sigma(Z_1))(\sigma(Z_0) Z_1)$, we can express $X$ as a smooth quadric surface in $\mathbb{P}^3_k$ defined by the equation:
\[
U_0 U_1 = U_2^2 - 5 U_3^2.
\]
\noindent \underline{\textit{Galois Action and the Automorphism Group.}}
Over the algebraic closure (or already over $L$), the base change $X_L$ splits into a product of the original variety and its Galois conjugate, yielding an isomorphism $X_L \cong \mathbb{P}^1_L \times \mathbb{P}^1_L$. Under this identification, the Galois group $\operatorname{Gal}(L/k)$ acts on $X_L$ not only by acting on the coordinates but also by swapping the two geometric factors (i.e., exchanging the two rulings of the quadric).

By the theory of Galois descent, the $k$-automorphism group $\operatorname{Aut}_k(X)$ consists of the $\operatorname{Gal}(L/k)$-invariant elements of $\operatorname{Aut}_L(X_L)$. The geometric swap action on the factors is Galois-invariant, and it descends to a $k$-automorphism. Consequently, we obtain the following group structure:
\[
\operatorname{Aut}_k(X) \cong \operatorname{PGL}_2(L) \rtimes \operatorname{Gal}(L/k).
\]
In this semi-direct product, the non-trivial element of $\operatorname{Gal}(L/k)$ acts on $\operatorname{PGL}_2(L)$ via Galois conjugation. By choosing the fields appropriately, the embedding of $A_5$ into $\operatorname{PGL}_2(L)$ combined with this Galois conjugation constructs a faithful $S_5$-action on $X$. Indeed, our specific choice of $k$ and $L$ in this example precisely realizes this phenomenon. However, in \cref{sec:non_closed}, we will impose an arithmetic condition on the base field that strictly excludes such exceptional fields. Consequently, under our main assumption, we do not need to consider $S_5$-actions, which completely rules out this non-split case.
\end{example}

\section{Four dimensional toric varieties with \texorpdfstring{$S_6$}{S6}-action over \texorpdfstring{$\C$}{C}}
\label{sec:dimension_4}

Throughout this section, we assume that the base field is the field of complex numbers $\C$.
The statement of \cite[Theorem 4.1]{esser2025symmetries} is incorrect in the case of dimension $n = 4$. It states that if $S_6$ acts faithfully on a complete simplicial toric variety $X$ of dimension $4$, then $X \cong \PP^4$ or $\PP^2 \times \PP^2$. In this section, we construct counterexamples and provide a complete classification.

In the proof of \cite[Theorem 4.1]{esser2025symmetries}, there are three cases: $S_k \cap \Autzero(X) = S_k$, $A_k$, or trivial. In the case $S_k \cap \Autzero(X) = S_k$, their argument contains a gap. The key lemma in this case is the following:

\begin{lemma}[{\cite[Lemma 4.4]{esser2025symmetries}}] \label{lem:EJM_dimension_bound}
    Let $X$ be a complete simplicial toric variety of dimension $n$. If $S_k \le \Autzero(X)$, then
    \[
        n \ge \begin{cases}
        3 & \text{if } k = 5,6, \\
        k-2 & \text{if } k \ge 7.
        \end{cases}
    \]
    Furthermore, if equality holds, then $X \cong \PP^n$. \qed
\end{lemma}

In the proof of \cite[Theorem 4.1]{esser2025symmetries}, it is concluded that if $S_k$ acts on $X$, where $k$ is the maximum degree of symmetric groups acting faithfully on an $n$-dimensional complete simplicial toric variety, then $X \cong \PP^n$ since equality in the above lemma holds. However, for $n = 4$, equality never holds.

Indeed, there is a trivial counterexample $\PP^3 \times \PP^1$. Since $S_6$ acts faithfully on $\PP^3$, $S_6$ also acts faithfully on $\PP^3 \times \PP^1$. The main result of this section is as follows:

\begin{theorem} \label{thm:main2}
    Let $X$ be a complete simplicial toric variety of dimension $4$. Then $S_6 \le \Autzero(X)$ if and only if $X$ is isomorphic to one of the following:
    \begin{itemize}
        \item $\PP^4$,
        \item $\PP(\mathcal{O}(a) \oplus \mathcal{O}) \to \PP^3$ for some even integer $a \in 2\Z$,
        \item $\PP(1,1,1,1,m)$ for some positive even integer $m \in 2\Z_{>0}$.
    \end{itemize}
\end{theorem}

The idea and the tools used in the proof of \cref{thm:main2} are similar to those in \cite{esser2025symmetries}. In the following discussion, we will adhere to the notation introduced in \cref{sec:preliminaries}.

First, we state several lemmas that will be used to prove \cref{thm:main2}.

\begin{lemma}[{\cite[Lemma 4.2]{esser2025symmetries}}] \label{lem:d_i_bound}
    With the notation as in \cref{sec:preliminaries},
    \begin{enumerate}[label={\rm (\arabic*)}]
        \item $\displaystyle \sum_{i=1}^s(d_i -1) \le 4$.
        \item If $\displaystyle \sum_{i=1}^s (d_i-1) = 4$, then $X \cong \PP^{d_1-1} \times \dots \times \PP^{d_s-1}$. \qed
    \end{enumerate}
\end{lemma}

\begin{lemma}[{\cite[Lemma 4.4]{esser2025symmetries}}] \label{lem:d1_geq_4}
    If $S_6 \le \Autzero(X)$, then $d_1 \ge 4$. \qed
\end{lemma}

\cref{lem:d_i_bound} is a special case of \cite[Lemma 4.2]{esser2025symmetries}. \cref{lem:d1_geq_4} follows from the proof of \cite[Lemma 4.4]{esser2025symmetries}. Now we prove \cref{thm:main2}.

\begin{proof}[Proof of \cref{thm:main2}]
    By \cref{lem:d_i_bound} (1) and \cref{lem:d1_geq_4}, we have $d_1 = 4$ or $d_1 = 5$. If $d_1 = 5$, then $X \cong \PP^4$ by \cref{lem:d_i_bound} (2). So we now assume that $d_1 = 4$. If $d_2 \ge 3$, we have $(d_1-1) + (d_2 -1) \ge 3 + 2 = 5 > 4$, and this contradicts \cref{lem:d_i_bound} (1). So $d_2 =0,1$, or $2$. If $d_2 = 0$, then $|\Delta(1)| = 4$. However, this contradicts the completeness of $\Delta$. If $d_2 = 2$, then $X \cong \PP^3 \times \PP^1$ by \cref{lem:d_i_bound} (2). Note that $\PP^3 \times \PP^1$ is a trivial projective $\PP^1$-bundle over $\PP^3$, i.e., $\PP^3 \times \PP^1 \cong \PP(\mathcal{O} \oplus \mathcal{O}) \to \PP^3$.

    So, the remaining case is that $(d_1,d_2) = (4,1)$. Let 
    \[
        \Delta(1) = \{\rho_1,\dots,\rho_4\} \sqcup \{\rho_5\} \sqcup \dots \sqcup \{\rho_d\}
    \]
    be the partition by the linear equivalence classes of the corresponding divisors. There exists an exact sequence \cite[\S3.4]{MR1234037}:
    \[
        0 \longrightarrow M \longrightarrow \bigoplus_{i=1}^d \Z \cdot D_i \longrightarrow \Cl(X) \longrightarrow 0,
    \]
    where the homomorphism from $M$ to $\bigoplus_{i=1}^d \Z \cdot D_i$ is defined by 
    \[
        u \mapsto \sum_{i=1}^d\langle u,v_i \rangle D_i.
    \]
    Now, for $j = 1,2,$ and $3$, $D_j \sim D_4$, and hence, there is $u_j \in M$ such that
    \[
        \sum_{i=1}^d\langle u_j,v_i \rangle D_i = D_j - D_4.
    \]
    So $u_j$ satisfies the following:
    \[
        \langle u_j,v_i \rangle = \begin{cases}
            1 & \text{if } i = j, \\
            -1 & \text{if } i = 4, \\
            0 & \text{otherwise}.
        \end{cases}
    \]
    
    Next, we claim that $u_1, u_2$, and $u_3$ are linearly independent in the vector space $M_\R$. Indeed, for $a_j \in \R \ (j = 1,2,3)$,
    \[
        a_1u_1 + a_2u_2 + a_3u_3 = 0 \ \implies \ {}^{\forall}j \in \{1,2,3\}, \  0 = \langle a_1u_1 + a_2u_2 + a_3u_3,v_j \rangle = a_j.
    \]
    
    Therefore, $u_1^{\perp} \cap u_2^{\perp} \cap u_3^{\perp} \subset N_\R$ has dimension $1$. For $j \ge 5$, 
    \[
        \langle u_1,v_j \rangle = \langle u_2,v_j \rangle = \langle u_3,v_j \rangle = 0
    \]
    by the construction of $u_j$. Thus, $v_j \in u_1^{\perp} \cap u_2^{\perp} \cap u_3^{\perp}$ for $j \ge 5$. Any $1$-dimensional subspace has at most two rays in it, and hence, $d = 5$ or $d = 6$. Furthermore, if $d = 6$, then $\rho_6 = -\rho_5$.

    \medskip \noindent \textit{Case 1}: First, consider the case $d = 5$, i.e., $\Delta(1) = \{\rho_1,\dots,\rho_4\} \sqcup \{\rho_5\}$. 

    Since $\Delta$ is complete, $-v_5$ is in some cone $\sigma \in \Delta$. If $\sigma$ contains $\rho_5$ as its ray, then $\sigma$ contains the $1$-dimensional vector space $\operatorname{span}\{v_5\}$, which contradicts the simpliciality of $\Delta$. Thus, we may assume that $\sigma$ is the cone generated by $v_1,v_2,v_3,v_4$. Thus, there exist $a_i \in \Z_{\ge 0}$ for $i = 1, \dots, 5$ such that $a_5 > 0$ and
    \[
        -a_5 v_5 = \sum_{i=1}^4a_i v_i.
    \]
    For $j = 1,2,3$,
    \[
        0 = \langle u_j, -a_5 v_5 \rangle = \left\langle u_j, \sum_{i=1}^4a_i v_i \right\rangle = a_j - a_4.
    \]
    Thus, $a_1 = a_2 = a_3 = a_4$, and this implies that $v_1 + v_2 + v_3 + v_4 = -m v_5$ for some positive integer $m$.

    Furthermore, $v_1,\dots,v_5$ generate the whole lattice $N$. Indeed, let $N'$ be the sublattice generated by $v_1,\dots,v_5$ and let $v$ be any element in $N$. Since $v_1,v_2,v_3$, and $v_5$ form a $\Q$-basis of $N_\Q$, $v = c_1v_1 + c_2v_2 + c_3v_3 + c_5v_5$ for some $c_1,c_2,c_3,c_5 \in \Q$. Now, $v \in N$ implies that $c_j = \langle u_j,v \rangle \in \Z$ for $j = 1,2,3$, and hence $c_1v_1 + c_2v_2 + c_3v_3 \in N' \subset N$. Therefore, $c_5v_5 \in N$, and by the minimality of $v_5$, we have $c_5 \in \Z$. This means that $v \in N'$.

    In conclusion, $X$ is a simplicial toric variety defined by a fan $\Delta$ generated by $v_1,\dots,v_5$ which generate $N$ and have the relation
    \[
        v_1 + v_2 + v_3 + v_4 + mv_5 = 0
    \]
    for some positive integer $m$. Such an $X$ is $\PP(1,1,1,1,m)$. We will show that $S_6$ acts faithfully on this weighted projective space if and only if $m$ is a positive even integer (see \cref{lem:wp}).

    \medskip \noindent \textit{Case 2}: Next, consider the case $d = 6$, i.e., $\Delta(1) = \{\rho_1,\dots,\rho_4\} \sqcup \{\rho_5\} \sqcup \{-\rho_5\}$.
    \begin{claim} \label{claim:linear_indep}
        If $d = 6$ and $(d_1,d_2,d_3)= (4,1,1)$, then $v_1, \dots, v_4$ are linearly independent in $N_\R$.
    \end{claim}
    We prove this claim first. If $v_1,\dots, v_4$ are linearly dependent, then $v_5 \notin \operatorname{span}\{v_1,\dots,v_4\}$ because $\Delta$ is complete. So $N_\R$ has the decomposition 
    \[
        N_\R = \operatorname{span}\{v_1,\dots, v_4\} \oplus \operatorname{span}\{v_5\}.
    \]
    Let $M_1$ be the dual lattice of $\operatorname{span}_\Z\{v_1,\dots,v_4\}$, and $M_2$ be the dual lattice of $\operatorname{span}_\Z\{v_5\}$, so $M = M_1 \oplus M_2$. By taking $u := (0,v_5^*) \in M_1 \oplus M_2 = M$, we have
    \[
        \sum_{i=1}^6\langle u,v_i \rangle D_i = \langle v_5^*,v_5 \rangle D_5 +\langle v_5^*,-v_5 \rangle D_6 = D_5 - D_6,
    \]
    which means $D_5 \sim D_6$, and contradicts the assumption. Therefore, $v_1,\dots,v_4$ are linearly independent, and the proof of the claim is completed.

    By the above claim, there exist $a_i \in \Z$ for $i = 1,\dots,5$ such that $a_5 >0$ and
    \[
        a_5v_5 = \sum_{i=1}^4a_i v_i.
    \]
    By the same argument as in the case $d = 5$, we get $a_1 = a_2 = a_3 = a_4$, and this implies that $v_1 + v_2 + v_3 + v_4 = av_5$ for some $a \in \Z$. Furthermore, $v_1,v_2,v_3,v_4,v_5,v_6$ generate $N$. In conclusion, $X$ is a complete simplicial toric variety defined by a fan $\Delta$ generated by $v_1,\dots,v_6$ with the relations
    \[
        v_1 + v_2 + v_3 + v_4 = av_5, \quad v_5 + v_6 = 0
    \]
    for some $a \in \Z$. Such an $X$ is exactly $\PP(\mathcal{O}(a) \oplus \mathcal{O}) \to \PP^3$. We will show that $S_6$ acts faithfully on this space if and only if $a$ is an even number (see \cref{lem:bdl}).
\end{proof}

\begin{lemma} \label{lem:wp}
    The symmetric group $S_6$ acts faithfully on $\PP(1,1,1,1,m)$ if and only if $m$ is a positive even integer.
\end{lemma}
\begin{proof}
    Let $X := \PP(1,1,1,1,m)$ and let $S = \C[x_0,\dots,x_4]$ be the weighted graded ring such that $X = \Proj S$. Since $X$ is well-formed (cf. \cite[\S 5]{esser2025symmetries}), we have a central extension
    \[
        1 \longrightarrow \C^* \longrightarrow \Aut(S) \longrightarrow \Aut(X) \longrightarrow 1,
    \]
    where $\C^*$ naturally acts by the weight, i.e., for $t \in \C^*$,
    \[
        t \cdot x_i = \begin{cases}
            tx_i & \text{if } i=0,1,2,3, \\
            t^mx_i & \text{if } i = 4.
        \end{cases}
    \]

    Since $\Aut(X)$ is the semidirect product
    \[
        \Aut(X) \cong R_U \rtimes (\operatorname{GL}_4(\C) \times \C^*)/\C^*,
    \]
    where $R_U$ is the unipotent radical of $\Aut(X)$ (\cite[Proposition 3.1]{MR4223973}) and $(\operatorname{GL}_4(\C) \times \C^*)/\C^*$ is obtained by dividing $\operatorname{GL}_4(\C) \times \C^*$ by the normal subgroup $\{(tI_4,t^m) \mid t \in \C^*\} \cong \C^*$, $S_6$ acts faithfully on $X$ if and only if there is an injective homomorphism 
    \[
        S_6 \hookrightarrow (\operatorname{GL}_4(\C) \times \C^*)/\C^*.
    \]

    So it is enough to show that there is an injection $\widetilde{S}_6 \hookrightarrow \operatorname{GL}_4(\C) \times \C^*$ that induces an injection $S_6 \hookrightarrow (\operatorname{GL}_4(\C) \times \C^*)/\C^*$ if and only if $m \in 2\Z_{>0}$.

    First, assume that $m \in 2\Z_{>0}$ and take a basic spin representation (cf. \cite{MR991411})
    \[
        \rho \colon \widetilde{S}_6 \hookrightarrow \operatorname{GL}_4(\C)
    \] 
    and a trivial representation 
    \[
        e \colon \widetilde{S}_6 \to \C^*.
    \]
    By taking the product, we obtain an injective homomorphism
    \[
        \rho \times e \colon \widetilde{S}_6 \hookrightarrow \operatorname{GL}_4(\C) \times \C^*.
    \]
    
    In order to verify that $\rho \times e$ induces an injective group homomorphism $S_6 \hookrightarrow (\operatorname{GL}_4(\C) \times \C^*)/\C^*$, it is enough to show that $(\rho \times e)^{-1}(\C^*) = \{\pm 1\}$. Since $(\rho \times e)^{-1}(\C^*)$ is the kernel of the composition of $\rho \times e$ with the natural projection to $(\operatorname{GL}_4(\C) \times \C^*)/\C^*$, it is a normal subgroup of $\widetilde{S}_6$. Furthermore, it is abelian since $\C^*$ is abelian. Now, 
    \[
        (\rho \times e)(-1) = (-I_4,1) = (-1\cdot I_4, (-1)^m) \in \C^*
    \]
    because $m \in 2\Z$. Therefore, $(\rho \times e)^{-1}(\C^*)$ is a nontrivial normal abelian subgroup of $\widetilde{S}_6$, and the only such subgroup is $\{\pm1\}$. 

    Next, assume that $m$ is an odd integer. Since the commutator subgroup of $\widetilde{S}_6$ is $2 \cdot A_6$, the double cover of $A_6$, and $\C^*$ is abelian, any representation of $\widetilde{S}_6$ of dimension $1$ must send $-1$ to $1$. Take any faithful representation $\rho \colon \widetilde{S}_6 \hookrightarrow \operatorname{GL}_4(\C)$ and any representation $\tau \colon \widetilde{S}_6 \to \C^*$ of degree $1$. If $\rho \times \tau$ induces a homomorphism from $S_6$ to $(\operatorname{GL}_4(\C)\times \C^*)/\C^*$, then $(\rho \times \tau)(-1)$ needs to be in $\C^*$. So there is some $t \in \C^*$ such that
    \[
        (\rho\times \tau)(-1) = (\rho(-1),1) = (tI_4,t^m).
    \]
    This implies $t^m = 1$ and $\rho(-1) = tI_4$. Now $(-1)^2 = 1$ in $\widetilde{S}_6$, and hence
    \[
        \rho(-1)^2 = t^2I_4 = I_4.
    \]
    Since $m$ is an odd number, $t^m = 1$ and $t^2 = 1$ imply $t = 1$, which means that $\rho(-1) = I_4$. This contradicts the faithfulness of $\rho$. Therefore, $\rho \times \tau$ does not induce 
    \[
        S_6 \hookrightarrow (\operatorname{GL}_4(\C)\times \C^*)/\C^*. \qedhere
    \]  
\end{proof}

\begin{lemma} \label{lem:bdl}
    The symmetric group $S_6$ acts faithfully on the projective bundle $\PP(\mathcal{O}(a) \oplus \mathcal{O}) \to \PP^3$ if and only if $a$ is an even integer.
\end{lemma}
\begin{proof}
    Let $X := \PP(\mathcal{O}(a) \oplus \mathcal{O}) \to \PP^3$. By the proof of \cref{thm:main2}, it is enough to show that $S_6 \le \Autzero(X)$ if and only if $a$ is an even integer. 
    
    First, assume that $S_6 \le \Autzero(X)$. Without loss of generality, we may assume $a \ge 0$. If $a = 0$, there is nothing to prove. For $a > 0$, let $Y$ be the weighted projective space $\PP(1,1,1,1,a)$. Since $X$ is the blow-up of $Y$ at a point, it follows from Blanchard's lemma (\cite[Corollary 7.2.2]{brion2015some}) that there is an injective homomorphism
    \[
        \Autzero(X) \hookrightarrow \Autzero(Y).
    \]
    By assumption, $S_6 \le \Autzero(Y)$, and hence $a \in 2\Z$ by \cref{lem:wp}.

    Conversely, assume that $a$ is an even number. It follows from \cref{lem:wp} that $S_6$ acts faithfully on $Y$. Since $X$ is the blow-up of $Y$ at its unique singular point, this $S_6$-action extends naturally to $X$.
\end{proof}

\section{Toric varieties over non-closed fields}
\label{sec:non_closed}

\subsection{Setup and Main Theorem}
\label{subsec:setup_main_thm}

In the previous sections, we considered complete simplicial toric varieties over the field of complex numbers $\C$. Combining the results of \cite[Theorem 4.1]{esser2025symmetries} with our \cref{thm:main2}, the maximum degree $m$ of the symmetric group $S_m$ acting faithfully on an $n$-dimensional complete simplicial toric variety over $\C$, and the classification of such varieties, are summarized in the following table:

\begin{table}[ht]
    \centering
    \renewcommand{\arraystretch}{1.2}
    \begin{tabular}{ccl}
        \hline
        Dimension & Max Degree & Varieties $X$ over $\C$ \\
        \hline
        $1$ & $S_4$ & $\PP^1$ \\
        $2$ & $S_5$ & $\PP^1 \times \PP^1$ \\
        $3$ & $S_6$ & $\PP^3$ \\
        $4$ & $S_6$ & $\PP^4$, $\PP^2 \times \PP^2$, $\PP(\mathcal{O}_{\PP_3} \oplus \mathcal{O}_{\PP_3}(2a))$, $\PP(1^4,2m)$ \\
        $n \ge 5$ & $S_{n+2}$ & $\PP^n$ \\
        \hline
    \end{tabular}
    \vspace{0.2cm}
    \caption{Maximal symmetric actions on toric varieties over $\C$.}
    \label{tab:complex_classification}
\end{table}

In this section, we transition from the algebraically closed field $\C$ to more general fields. Let $k$ be a field of characteristic zero. The representation theory of finite groups over $k$ strongly depends on whether certain representations can be realized over $k$. Specifically, the existence of certain representations of symmetric and alternating groups imposes arithmetic conditions on the base field. 

\vspace{0.5em}
\noindent\begin{minipage}{\textwidth}
\hspace*{1.5em}Throughout this section, we assume that the base field $k$ satisfies the following condition:

\vspace{0.5em}
\phantomsection\label{cond:star}
\def\@currentlabel{($\star$)}
\begin{itemize}[leftmargin=5.5em]
    \item[\textbf{($\star$)}\quad{\rm (1)}] The characteristic of $k$ is zero.
    \item[\hphantom{\textbf{($\star$)}\quad}{\rm (2)}] The equation $a^2 + b^2 = -1$ has no solution for $a, b \in k$.
    \item[\hphantom{\textbf{($\star$)}\quad}{\rm (3)}] $\sqrt{5} \in k$, or $k(\sqrt{5})$ also satisfies the condition {\rm(2)}.
\end{itemize}
\end{minipage}
\vspace{0.5em}

The second requirement is exactly the necessary and sufficient condition under which the symmetric group $S_4$ cannot be embedded into the projective linear group $\operatorname{PGL}_2(k)$. Typical examples of fields satisfying the condition \hyperref[cond:star]{($\star$)} include the field of rational numbers $\Q$ and the field of real numbers $\R$. Note that this condition naturally implies $\sqrt{-3} \notin k$, which will be crucial in our subsequent classification (\cref{subsec:dim3}).

Under this arithmetic condition, the allowed symmetries are significantly restricted compared to the complex case. The main result of this paper is the following classification over such fields, which highlights a striking rigidity in dimensions $n=1$ and $n \ge 3$, while revealing a rich structure in dimension $n=2$.

\begin{theorem}[Main Theorem] \label{thm:main_nonclosed}
    Let $k$ be a field satisfying the condition \hyperref[cond:star]{\rm{($\star$)}}. Let $m$ be the maximum degree of the symmetric group $S_m$ acting faithfully on an $n$-dimensional complete simplicial toric variety $X$ over $k$. Then $m$ and the isomorphism classes of $X$ are given as follows:
    
    \begin{table}[ht]
        \centering
        \renewcommand{\arraystretch}{1.2}
        \begin{tabular}{ccl}
            \hline
            Dimension & Max Degree & Varieties $X$ over $k$ \\
            \hline
            $1$ & $S_3$ & $\PP^1$ \\
            $2$ & $S_4$ & Infinitely many surfaces (see \cref{prop:dim2_classification}) \\
            $n \ge 3$ & $S_{n+2}$ & $\PP^n$ \\
            \hline
        \end{tabular}
        \vspace{0.2cm}
        \caption{Maximal symmetric actions on toric varieties over $k$.}
        \label{tab:nonclosed_classification}
    \end{table}
\end{theorem}

In the subsequent subsections, we will prove \cref{thm:main_nonclosed} step by step according to the dimension $n$.

\subsection{Cases \texorpdfstring{$n = 1$}{n=1} and \texorpdfstring{$n \ge 5$}{n5}}
\label{subsec:dim1_and_5}

In this subsection, we treat the cases of dimension $n=1$ and $n \ge 5$. These cases can be handled straightforwardly by combining the arithmetic condition on the base field with a base change argument to the algebraic closure.

\begin{proposition} \label{prop:dim1_and_5}
    Let $k$ be a field satisfying the condition \hyperref[cond:star]{\rm{($\star$)}}. Let $X$ be an $n$-dimensional complete simplicial toric variety over $k$. If $n \notin \{2,3,4\}$ and the symmetric group $S_m$ acts faithfully on $X$, then $m \le n+2$. Furthermore, if $m = n+2$, then $X \cong \PP^n_k$.
\end{proposition}

\begin{proof}
    First, we consider the case $n=1$. The only complete toric variety of dimension $1$ is the projective line $X \cong \PP^1_k$, and its automorphism group is exactly the projective linear group $\operatorname{PGL}_2(k)$. By the condition \hyperref[cond:star]{\rm{($\star$)}}, the equation $a^2 + b^2 = -1$ has no solution in $k$. It is a classical result (see, e.g., \cite{beauville2009finite}) that this is precisely the condition under which the symmetric group $S_4$ cannot be embedded into $\operatorname{PGL}_2(k)$. Thus, $S_4$ cannot act faithfully on $\PP^1_k$, meaning $m < 4$. On the other hand, the symmetric group $S_3 \cong D_3$ naturally embeds into $\operatorname{PGL}_2(k)$ for any field $k$ of characteristic zero. Therefore, the maximum degree is $m=3$, which is realized by $\PP^1_k$.

    Next, we consider the case $n \ge 5$. The symmetric group $S_{n+2}$ admits an $(n+1)$-dimensional standard representation defined over $\Q$, and thus over $k$. Projectivizing this representation gives a faithful linear action of $S_{n+2}$ on $\PP^n_k$, which shows that the maximum degree $m$ is at least $n+2$.
    
    To prove the upper bound and the uniqueness, assume that $S_m$ acts faithfully on an $n$-dimensional complete simplicial toric variety $X$ over $k$. Let $\overline{k}$ be an algebraic closure of $k$. Then the base change $X_{\overline{k}} := X \times_{\Spec k} \Spec \overline{k}$ is an $n$-dimensional complete simplicial toric variety over $\overline{k}$, and the $S_m$-action naturally extends to $X_{\overline{k}}$. Since $\overline{k}$ is an algebraically closed field of characteristic zero, the structural results and the classification over $\C$ apply identically to $X_{\overline{k}}$. According to \cref{tab:complex_classification}, for $n \ge 5$, the maximum degree of a symmetric group acting on an $n$-dimensional complete simplicial toric variety is $n+2$. This implies that $m \le n+2$.

    Furthermore, if the equality $m = n+2$ holds, the classification over the algebraic closure dictates that $X_{\overline{k}} \cong \PP^n_{\overline{k}}$ (i.e., $X$ is a Severi--Brauer variety). Since $X$ contains a rational point, we have $X \cong \mathbb{P}_k^n$ by \cite[Theorem 5.1.3]{GS06}.
\end{proof}

\subsection{Case \texorpdfstring{$n = 3$}{n=3} and \texorpdfstring{$n=4$}{n=4}}
\label{subsec:dim3}

In this subsection, we completely classify 3 and 4 dimensional complete simplicial toric varieties admitting maximal symmetric actions over $k$. According to the classification over $\C$ (see \cref{tab:complex_classification}), the maximum degree of the symmetric group acting on a $3$-fold is $6$, which is realized uniquely by $\PP^3$. However, over a field $k$ satisfying the condition \hyperref[cond:star]{\rm{($\star$)}}, the $S_6$-action is obstructed.

Before proceeding to the proof, we classify 3-dimensional complete simplicial toric varieties over $\mathbb{C}$ admitting an $S_5$-action.

\begin{lemma}\label{lemma:n3}
    Let $X$ be a 3-dimensional complete simplicial toric variety over $\C$. Then $S_5$ acts faithfully on $X$ if and only if $X$ is isomorphic to one of the following:
    \begin{itemize}
        \item $\PP^3$,
        \item $\PP(\mathcal{O}(a,a) \oplus \mathcal{O}) \to \PP^1 \times \PP^1$ for some integer $a \in \mathbb{Z}$.
    \end{itemize}
\end{lemma}

\begin{proof}
    Suppose that $S_5$ acts faithfully on a $3$-dimensional complete simplicial toric variety $X$. We consider the action on the Cox ring and the induced partition of the rays $\Delta(1) = \bigsqcup_{i=1}^s \Delta_i$, where $d_i = |\Delta_i|$ and we order them as $d_1 \ge d_2 \ge \dots \ge d_s$. 
    
    By \cite[Lemma 4.5]{esser2025symmetries}, the alternating group $A_5$ must be contained in the connected component, so $A_5 \le \operatorname{Aut}^0(X)$. Since $A_5$ does not admit any non-trivial $1$-dimensional projective representation, we must have $d_1 \ge 2$. We examine the cases based on the maximum block size $d_1$:

    \noindent \textbullet\ \textit{Case $d_1 = 4$:} The sum of the dimensions is $(4-1) = 3 = \dim X$. This implies there is only one block, uniquely identifying the variety as $X \cong \PP^3$.

    \noindent \textbullet\ \textit{Case $d_1 = 3$:} The symmetric group $S_5$ cannot be a subgroup of $\operatorname{Aut}^0(X)$ because the minimum degree for a faithful projective representation of $S_5$ is $4$ (i.e., $S_5 \le \operatorname{PGL}_4$). Thus, the $S_5$-action must induce an outer automorphism. By \cite[Lemma 4.6]{esser2025symmetries}, the partition must be $(d_1, d_2) = (3, 2)$. According to \cite[Lemma 4.2]{esser2025symmetries}, this corresponds to the product $X \cong \PP^2 \times \PP^1$. However, $S_5$ cannot act faithfully on $\PP^2 \times \PP^1$, which yields a contradiction.

    \noindent \textbullet\ \textit{Case $d_1 = 2$:} By \cite[Lemma 4.2]{esser2025symmetries}, the requirement of the $A_5$-action forces the presence of at least two blocks of size $2$, meaning $(d_1, d_2) = (2, 2)$. We then consider the remaining rays:
    \begin{itemize}
        \item If $d_3 = 2$, the partition is $(2, 2, 2)$, which corresponds to $X \cong \PP^1 \times \PP^1 \times \PP^1$. This is precisely the case of $\PP(\mathcal{O}(a,a) \oplus \mathcal{O})$ where $a = 0$. Since $S_5$ acts on $\PP^1 \times \PP^1$, this diagonal action naturally extends to the three-fold product.
        \item If $d_3 = 1$, following the exact same argument as in \cref{thm:main2}, the rays of the fan $\Delta$ are essentially given by $(1,0,a)$, $(-1,0,0)$, $(0,1,b)$, $(0,-1,0)$, $(0,0,1)$, and $(0,0,-1)$, or a simplicial subdivision obtained by omitting $(0,0,1)$. Specifically, the two blocks of size $2$ are exactly $\Delta_1 = \{ (1,0,a), (-1,0,0) \}$ and $\Delta_2 = \{ (0,1,b), (0,-1,0) \}$.

        Since $S_5 \not\le \operatorname{Aut}^0(X)$, the $S_5$-action must involve a $\mathbb{Z}/2\mathbb{Z}$-action on the lattice $N$ that swaps the blocks $\Delta_1$ and $\Delta_2$. To perform this swap while fixing the $z$-axis (i.e., the subspace containing $(0,0,\pm 1)$), the lattice automorphism must strictly exchange the first and second coordinates. Applying this coordinate exchange to the ray $(1,0,a)$ yields $(0,1,a)$. For this ray to belong to $\Delta_2$, it must coincide with $(0,1,b)$, which immediately forces $a=b$.

        Furthermore, we must consider the constraint imposed by the simplicial condition. The four rays in $\Delta_1 \cup \Delta_2$ generate a quadrangular pyramid in $N_\R$. If the central ray $(0,0,1)$ is omitted, this quadrangular cone must be subdivided into two triangular cones by choosing a diagonal to ensure the fan remains simplicial. There are exactly two possible choices for this diagonal. However, the coordinate-swapping automorphism acts as a reflection across the plane $x=y$, which exchanges these two diagonals. Consequently, the swap cannot preserve any such subdivision, and the $S_5$-action would fail to preserve the fan. Thus, the case without $(0,0,1)$ is completely excluded.
    \end{itemize}

    Conversely, for any variety of the form $\PP(\mathcal{O}(a,a) \oplus \mathcal{O}) \to \PP^1 \times \PP^1$, the faithful $S_5$-action on the base $\PP^1 \times \PP^1$ naturally lifts to the total space, preserving the symmetric bundle structure. This completes the classification.
\end{proof}

\begin{remark}
    The assumption that the toric variety $X$ is simplicial is strictly necessary for this classification. As observed in the proof, the $S_5$-action on the $5$-ray configuration fails to preserve the fan solely because the quadrangular pyramid must be subdivided into simplicial cones, which inevitably breaks the reflection symmetry. If we relax the simplicial condition and allow the quadrangular pyramid to remain as a single $3$-dimensional maximal cone, this symmetry-breaking does not occur. Consequently, if we permit non-simplicial varieties, there indeed exist complete toric $3$-folds defined by exactly $5$ rays that admit a faithful $S_5$-action.
\end{remark}

\begin{proposition} \label{prop:dim3}
    Let $k$ be a field satisfying the condition \hyperref[cond:star]{\rm{($\star$)}}. Let $X$ be a $3$-dimensional complete simplicial toric variety over $k$. If the symmetric group $S_m$ acts faithfully on $X$, then $m \le 5$. Furthermore, if $m = 5$, then $X \cong \PP^3_k$.
\end{proposition}

\begin{proof}
    First, we show that $S_6$ cannot act faithfully on any $3$-dimensional complete simplicial toric variety $X$ over $k$. Suppose for contradiction that $S_6 \le \Aut_k(X)$. By the same base change argument as in \cref{prop:dim1_and_5}, $X_{\overline{k}} \cong \PP^3_{\overline{k}}$, and the action induces an embedding $S_6 \hookrightarrow \operatorname{PGL}_4(k)$. This projective representation lifts to a linear representation of the double cover $\Gamma = \tilde{S}_6$ into $\operatorname{GL}_4(k)$. Over $\C$, this corresponds to a $4$-dimensional basic spin representation of $\tilde{S}_6$. However, from the theory of shifted tableaux \cite[Theorem 3.3]{MR991411}, the character $\varphi$ of this spin representation evaluated on a lift of a $6$-cycle $\sigma$ yields $\varphi(\sigma) = \pm\sqrt{-3}$. For this representation to be realized over the base field $k$, its character values must lie in $k$, which requires $\sqrt{-3} \in k$.
    
    Remarkably, the condition \hyperref[cond:star]{\rm{($\star$)}} strictly prevents this. If $\sqrt{-3} \in k$, then one can explicitly write $-1$ as a sum of two squares in $k$:
    \[
        \left( \frac{1+\sqrt{-3}}{2} \right)^2 + \left( \frac{1-\sqrt{-3}}{2} \right)^2 = \frac{-2+2\sqrt{-3}}{4} + \frac{-2-2\sqrt{-3}}{4} = -1.
    \]
    Since the condition \hyperref[cond:star]{\rm{($\star$)}} assumes that $a^2 + b^2 = -1$ has no solution in $k$, it follows that $\sqrt{-3} \notin k$. Thus, this $4$-dimensional representation cannot be realized over $k$, which contradicts the embedding $\tilde{S}_6 \hookrightarrow \operatorname{GL}_4(k)$. Hence, $m < 6$.
    
    Next, the symmetric group $S_5$ naturally acts on $\PP^3_k$ via the standard $4$-dimensional representation (which is defined over $\Q$ and thus over $k$). Therefore, the maximum degree is exactly $m = 5$.

    We now classify the varieties realizing this maximum degree. Suppose that $S_5$ acts faithfully on $X$ over $k$, which means $S_5 \le \operatorname{Aut}_k(X)$. By \cref{lemma:n3}, if we temporarily base change to the algebraic closure $\overline{k}$, the geometric candidates for $X_{\overline{k}}$ are restricted to $X_{\overline{k}} \cong \mathbb{P}^3_{\overline{k}}$ or $X_{\overline{k}} \cong \mathbb{P}(\mathcal{O}(a,a) \oplus \mathcal{O}) \to \PP^1_{\overline{k}} \times \PP^1_{\overline{k}}$. 

    If $X_{\overline{k}} \cong \mathbb{P}^3_{\overline{k}}$, then $X$ is a Severi-Brauer variety of dimension $3$. By \cite[Theorem 5.1.3]{GS06}, we conclude that $X \cong \mathbb{P}^3_k$. 

    In the following, we consider the case where $X_{\overline{k}} \cong \mathbb{P}(\mathcal{O}(a,a) \oplus \mathcal{O})$. We divide the analysis into the split and non-split cases.

    \noindent \textit{The Split Case.} 
    Suppose that $X$ is a split toric variety, meaning $X \cong \mathbb{P}(\mathcal{O}(a,a) \oplus \mathcal{O})$ over $k$. In this case, the Weil divisor class group $\operatorname{Cl}(X)$ is a torsion-free finitely generated abelian group, i.e., $\operatorname{Cl}(X) \cong \mathbb{Z}^r$. 

    The absence of torsion in the class group implies that the characteristic group $G = \operatorname{Hom}(\operatorname{Cl}(X), \mathbb{G}_m)$ is a split torus $\mathbb{G}_{m,k}^r$ over the base field $k$. By Hilbert's Theorem 90 for split tori (cf. \cite[Chapter II, Proposition 1]{serre1997galois}), the first Galois cohomology vanishes: $H^1(k, G_{\overline{k}}) = 0$.

    The toric variety $X$ is constructed as a geometric quotient of the spectrum of its Cox ring $R$ by the action of $G$. The exact sequence of algebraic group schemes $1 \to G_{\overline{k}} \to \operatorname{Aut}_g(R)_{\overline{k}} \to \operatorname{Aut}^0(X_{\overline{k}}) \to 1$ induces a long exact sequence of $k$-rational points. Since $H^1(k, G_{\overline{k}}) = 0$, the projection is surjective:
    $$ \operatorname{Aut}_g(R) \twoheadrightarrow \operatorname{Aut}^0_k(X). $$
    Thus, any $k$-automorphism of $X$ lifts to a graded $k$-algebra automorphism of the Cox ring $R$. The subgroup $A_5 \le \operatorname{Aut}^0_k(X)$ therefore lifts to a central extension $\Gamma \le \operatorname{Aut}_g(R)$.

    We have already established that the Cox ring variables for this specific geometric structure correspond to the block sizes $(d_1, d_2, d_3, d_4) = (2, 2, 1, 1)$ or $(d_1,d_2,d_3) = (2,2,2)$. The graded automorphism group $\operatorname{Aut}_g(R)$ decomposes as a semi-direct product of its unipotent radical $R_u(k)$ and the linear reductive part $\prod_i \operatorname{GL}_{d_i}(k)$.

    We consider the natural projection $p \colon \operatorname{Aut}_g(R) \twoheadrightarrow \prod_i \operatorname{GL}_{d_i}(k)$. The kernel of the restriction $p|_{\Gamma}$ is $\Gamma \cap R_u(k)$. Since $k$ is of characteristic zero, every non-trivial element of $R_u(k)$ has infinite order, whereas $\Gamma$ is finite. Thus, the intersection is trivial ($\Gamma \cap R_u(k) = \{1\}$), inducing a faithful embedding:
    $$ \Gamma \hookrightarrow \prod_i \operatorname{GL}_{d_i}(k)$$
    Projecting $\Gamma$ onto any $\operatorname{GL}_2(k)$ block would induce an embedding of $A_5$ into $\operatorname{PGL}_2(k)$. However, by our condition \hyperref[cond:star]{\rm{($\star$)}}, $A_5$ cannot embed into $\operatorname{PGL}_2(k)$. This yields a contradiction, ruling out the split case.

    \noindent \textit{The Non-split Case.} 
    Suppose $X$ is a non-split $k$-form of $\mathbb{P}(\mathcal{O}(a,a) \oplus \mathcal{O})$. The Galois group $\operatorname{Gal}(\overline{k}/k)$ acts non-trivially on the rays of the fan. The symmetry group of the rays of the base $\mathbb{P}^1 \times \mathbb{P}^1$ is the dihedral group $D_8$. Swapping rays within the same block yields a $k$-form that is isomorphic up to a torus action. Therefore, if $a \neq 0$, we only need to consider the action of a quadratic extension $L/k$ that specifically swaps the two blocks (i.e., the two $\mathbb{P}^1$ factors).

    This action corresponds to a Weil restriction, giving the $k$-form $X \cong \mathbb{P}(\mathcal{O}_B \oplus \mathcal{O}_B(a,a))$ where $B = \operatorname{Res}_{L/k}(\mathbb{P}^1_L)$. As discussed in \cref{subsec:weil_res}, the automorphism group of this Weil restriction contains the following semi-direct product:
    $$ \operatorname{Aut}_k(X) \supset \operatorname{PGL}_2(L) \rtimes \operatorname{Gal}(L/k). $$
    For $S_5$ to act faithfully on $X$, we must have an embedding $A_5 \le \operatorname{PGL}_2(L)$ and the structural relation $A_5 \rtimes \operatorname{Gal}(L/k) \cong S_5$.

    If such a representation of $A_5$ into $\operatorname{PGL}_2(L)$ exists, the characters of the two non-conjugate elements of order $5$ necessarily involve $\sqrt{5}$. For the Galois action to induce the outer automorphism that extends $A_5$ to $S_5$, the non-trivial element of $\operatorname{Gal}(L/k)$ must swap these characters, which requires mapping $\sqrt{5}$ to $-\sqrt{5}$. This forces the quadratic extension to be $L = k(\sqrt{5})$. 

    This exhausts all possibilities and completely rules out the non-split case for $a \neq 0$.

    However, by the condition \hyperref[cond:star]{\rm{($\star$)}}, $a^2 + b^2 = -1$ has no solutions in $k(\sqrt{5})$. Consequently, $A_5$ cannot be embedded into $\operatorname{PGL}_2(k(\sqrt{5}))$. This contradiction completely rules out the non-split and $a\neq 0$ case as well.

    If $a=0$, then $X_{\overline{k}} \cong \mathbb{P}^1_{\overline{k}} \times \mathbb{P}^1_{\overline{k}} \times \mathbb{P}^1_{\overline{k}}$. In this case, all three blocks have size $2$, meaning the geometric symmetry includes $S_3$ permuting the three $\mathbb{P}^1$ factors. The Galois action on these factors determines the $k$-form. There are two possibilities for the orbit partition of the Galois action:
    \begin{itemize}
        \item The orbit partition is $2 + 1$. In this case, the splitting field for the permuted factors is a quadratic extension $L/k$. The $k$-form is given by $X \cong \operatorname{Res}_{L/k}(\mathbb{P}^1_L) \times \mathbb{P}^1_k$, and the continuous part of its automorphism group is $\operatorname{PGL}_2(L) \times \operatorname{PGL}_2(k)$. This essentially mirrors the argument for the $a \neq 0$ case: realizing the $S_5$-action would require an embedding $A_5 \le \operatorname{PGL}_2(L)$ such that the Galois conjugation induces an outer automorphism. This forces $L = k(\sqrt{5})$, but the condition \hyperref[cond:star]{\rm ($\star$)} prevents $A_5$ from embedding into $\operatorname{PGL}_2(k(\sqrt{5}))$, yielding the exact same contradiction.
        
        \item The orbit partition is $3$. This means the Galois group contain an element of order three. Let $L$ be the splitting field of $X$ over $k$, so that $X_L \cong \mathbb{P}^1_L \times \mathbb{P}^1_L \times \mathbb{P}^1_L$. The image of the Galois group $\operatorname{Gal}(L/k)$ in $S_3$ must contain a cyclic subgroup of order $3$. Let $k'$ be the intermediate field corresponding to this subgroup, so that $L/k'$ is a cyclic Galois extension of degree $3$ with $\operatorname{Gal}(L/k') \cong \mathbb{Z}/3\mathbb{Z}$.
        
        Over the field $k'$, the Galois group cyclically permutes the three factors, meaning the $k'$-form is exactly the Weil restriction $X_{k'} \cong \operatorname{Res}_{L/k'}(\mathbb{P}^1_L)$. The automorphism group over $k'$ is:
        \[
            \operatorname{Aut}_{k'}(X_{k'}) \cong \operatorname{PGL}_2(L) \rtimes \operatorname{Gal}(L/k') \cong \operatorname{PGL}_2(L) \rtimes \mathbb{Z}/3\mathbb{Z}. 
        \]
        Since $X$ admits a faithful $S_5$-action over $k$, this action naturally restricts to a faithful $S_5$-action on $X_{k'}$ over $k'$. Thus, we must have an embedding:
        \[
            S_5 \hookrightarrow \operatorname{PGL}_2(L) \rtimes \mathbb{Z}/3\mathbb{Z}. 
        \]
        Consider the projection map from this semi-direct product to $\mathbb{Z}/3\mathbb{Z}$. Since the only normal subgroups of $S_5$ are $\{1\}$, $A_5$, and $S_5$, the image of $S_5$ under any homomorphism to $\mathbb{Z}/3\mathbb{Z}$ must be trivial. This implies that the entire group $S_5$ must embed into the kernel, namely $\operatorname{PGL}_2(L)$. However, the maximal symmetric group that can be embedded into $\operatorname{PGL}_2$ over any field of characteristic zero is $S_4$. Therefore, the embedding $S_5 \le \operatorname{PGL}_2(L)$ is impossible. This yields a direct contradiction.
    \end{itemize}

    Therefore, we conclude that $\mathbb{P}^3_k$ is the unique $3$-dimensional complete simplicial toric variety admitting a faithful $S_5$-action over $k$.
\end{proof}

We briefly extend our argument to dimension $4$. The strategy is entirely analogous to the $n=3$ case, exploiting the geometric and representation-theoretic rigidities over the base field $k$. According to the complex classification, the maximum degree of a symmetric group acting on a $4$-dimensional complete simplicial toric variety is $6$, realized by $\PP^4$. We show that this remains true and is the unique maximal case over $k$.

\begin{lemma} \label{lemma:trivial_forms}
Let $X$ be a $k$-form for $\mathbb{P}^4$, $\mathbb{P}(\mathcal{O}(2a) \oplus \mathcal{O})$, or $\mathbb{P}(1,1,1,1,2m)$. Then $X$ is isomorphic to a split toric variety over $k$.
\end{lemma}

\begin{proof}
For $\mathbb{P}^4$, any $k$-form is a Severi-Brauer variety of dimension $4$. Since $X$ is a toric variety, it contains a $k$-rational point (the identity element of the algebraic torus). By \cite[Theorem 5.1.3]{GS06}, the existence of a rational point implies that $X$ splits completely, giving $X \cong \mathbb{P}^4_k$.

For $\mathbb{P}(\mathcal{O}(2a) \oplus \mathcal{O})$ and $\mathbb{P}(1,1,1,1,2m)$, we consider the action of the Galois group $\operatorname{Gal}(\overline{k}/k)$ on the rays of their fans. For $\mathbb{P}(\mathcal{O}(2a) \oplus \mathcal{O})$, the ray generators satisfy the relations $v_1 + v_2 + v_3 + v_4 + 2av_5 = 0$ and $v_5 + v_6 = 0$. Because the Galois action must preserve these linear relations, it cannot map rays with different weights to each other. Therefore, the Galois group strictly preserves the structural blocks of the rays, meaning it can only permute rays within the same block. 

The same logic applies to $\mathbb{P}(1,1,1,1,2m)$, where the relation is $v_1 + v_2 + v_3 + v_4 + 2mv_5 = 0$ and the ray $v_5$ is completely fixed. A $k$-form arising solely from intra-block permutations is geometrically equivalent to a change of variables within the homogeneous coordinates of the Cox ring. Up to the torus action, such forms are naturally isomorphic to the split toric variety over $k$.
\end{proof}

\begin{proposition} \label{prop:dim4}
    Let $k$ be a field satisfying the condition \hyperref[cond:star]{\rm{($\star$)}}. Let $X$ be a $4$-dimensional complete simplicial toric variety over $k$. If the symmetric group $S_m$ acts faithfully on $X$, then $m \le 6$. Furthermore, if $m = 6$, then $X \cong \PP^4_k$.
\end{proposition}

\begin{proof}
    The symmetric group $S_6$ naturally acts on $\PP^4_k$ via the standard $5$-dimensional representation (defined over $\Q$). Thus, the maximum degree is at least $6$. If $S_7 \le \Aut_k(X)$, the base change to $\overline{k}$ would force $S_7 \le \Aut(X_{\overline{k}})$, but $S_7$ does not admit any faithful action over an algebraically closed field of characteristic zero, ruling out $m \ge 7$. Hence, the maximum degree is exactly $m = 6$.
    
    We now classify the varieties realizing this maximum degree. Suppose that $S_6$ acts faithfully on $X$ over $k$, which implies $S_6 \le \operatorname{Aut}_k(X)$. By base change and the classification of maximum symmetries, the candidates for $X_{\overline{k}}$ are limited to $\mathbb{P}^4_{\overline{k}}$, $(\mathbb{P}^2 \times \mathbb{P}^2)_{\overline{k}}$, $\mathbb{P}(\mathcal{O}(2a) \oplus \mathcal{O})_{\overline{k}}$, and $\mathbb{P}(1,1,1,1,2m)_{\overline{k}}$.

\noindent \textit{The Split Case.}
Suppose $X$ is a split toric variety over $k$. Similar to the $3$-dimensional case, the subgroup $A_6 \le \operatorname{Aut}^0_k(X)$ lifts to a central extension $\Gamma \le \prod_{i=1}^s \operatorname{GL}_{d_i}(k)$. Under the base field condition \hyperref[cond:star]{\rm{($\star$)}}, realizing this faithful representation requires a block size of at least $d_i \ge 4$ (specifically $d_1 = 5$ for $\mathbb{P}^4$). This dimensional requirement strictly rules out the other split candidates, restricting the geometry to $X \cong \mathbb{P}^4_k$.

\noindent \textit{The Non-split Case.}
Suppose $X$ is a non-split toric variety. By \cref{lemma:trivial_forms}, the only geometric candidate that can yield a non-trivial $k$-form not isomorphic to a split toric variety is $\mathbb{P}^2 \times \mathbb{P}^2$. A non-split form of $\mathbb{P}^2 \times \mathbb{P}^2$ must involve a $\mathbb{Z}/2\mathbb{Z}$ action that swaps the two structural blocks (the $\mathbb{P}^2$ factors). Thus, the variety is isomorphic to the Weil restriction $X = \operatorname{Res}_{L/k}(\mathbb{P}^2_L)$ for some quadratic extension $L/k$.

The automorphism group of this Weil restriction is given by:
\[
\operatorname{Aut}_k(X) \cong \operatorname{PGL}_3(L) \rtimes \operatorname{Gal}(L/k).
\]
For an $S_6$-action to exist, we must have an embedding $A_6 \le \operatorname{PGL}_3(L)$, and the Galois group $\operatorname{Gal}(L/k)$ must induce an outer automorphism extending $A_6$ to $S_6$.

According to \cite{conway1985atlas}, the projective representation $A_6 \le \operatorname{PGL}_3(\overline{k})$ is induced by the linear representation of the central extension $3.A_6 \le \operatorname{SL}_3(\overline{k})$. In this representation, the characters of the two non-conjugate elements of order $5$ involve the values $\frac{1 - \sqrt{5}}{2}$ and $\frac{1 + \sqrt{5}}{2}$. The Galois action must swap these elements to induce the necessary outer automorphism, which forces $\operatorname{Gal}(L/k)$ to map $\sqrt{5}$ to $-\sqrt{5}$. Hence, we must have $L = k(\sqrt{5})$.

Furthermore, the group $3.A_6$ possesses a central element of order $3$. In the $3$-dimensional representation, this element acts as the scalar matrix $\omega I$ where $\omega = \frac{-1+\sqrt{-3}}{2}$, strictly requiring that $\sqrt{-3} \in L = k(\sqrt{5})$. However, by the condition \hyperref[cond:star]{\rm{($\star$)}}, $\sqrt{-3} \notin k$, which further implies $\sqrt{-3} \notin k(\sqrt{5})$. This presents a direct contradiction.

Therefore, $S_6$ cannot act on the non-split form $\operatorname{Res}_{L/k}(\mathbb{P}^2_L)$. We conclude that $\mathbb{P}^4_k$ is the unique $4$-dimensional complete simplicial toric variety admitting a faithful $S_6$-action over $k$.
\end{proof}

\subsection{Case \texorpdfstring{$n = 2$}{n=2}}
\label{subsec:dim2}

We now turn to the $2$-dimensional case, where the geometric behavior diverges significantly from higher dimensions. Over the algebraic closure $\overline{k}$, the maximal symmetric action on a toric surface is $S_5$, which is realized uniquely on $\PP^1_{\overline{k}} \times \PP^1_{\overline{k}}$. 

As discussed in the proof of \cref{prop:dim3}, the $k$-forms for $\PP^1 \times \PP^1$ are exactly the split form $\PP^1_k \times \PP^1_k$ and the Weil restriction $\operatorname{Res}_{L/k}(\PP^1_L)$ for some quadratic extension $L/k$. Due to the condition \hyperref[cond:star]{\rm{($\star$)}}, neither of these forms admits a faithful $S_5$-action. Consequently, the maximum degree of a symmetric group acting on a toric surface over $k$ is at most $4$.

Remarkably, unlike the higher-dimensional cases where the maximal action uniquely determines the variety, an $S_4$-action can be realized on infinitely many toric surfaces, including singular and non-split ones. This infinite family arises from a natural lattice construction.

\begin{lemma} \label{lem:S4_action_surface}
    Let $N_1 = \{(x,y,z) \in \Z^3 \mid x+y+z=0\}$ and $N_2 = \Z^3/\Z(1,1,1)$ be lattices of rank $2$. The symmetric group $S_3$ naturally acts on $N_i$ by permuting the coordinates. Let $\Delta$ be a complete fan in $(N_i)_{\mathbb{R}}$ that is invariant under this $S_3$-action. Then the associated complete split toric surface $X(\Delta)$ over $k$ admits a faithful $S_4$-action.
\end{lemma}

\begin{proof}
    Let $T \cong \GG_{m,k}^2$ be the algebraic torus associated with $N$. Since $\operatorname{char} k = 0$, the $k$-rational points of $T$ contain the $2$-torsion subgroup $T[2] \cong N \otimes_\Z \{\pm 1\} \cong (\Z/2\Z)^2 \cong V_4$, where $V_4$ is the Klein four-group. Since any toric variety admits a faithful action of its own torus, $V_4$ acts faithfully on $X(\Delta)$.
    
    On the other hand, the $S_3$-action on $N_i$ preserves the fan $\Delta$, so it induces an action on $X(\Delta)$ by toric automorphisms. This $S_3$-action normalizes the torus $T$ and its $2$-torsion subgroup $V_4$. Thus, we obtain an action of the semi-direct product $V_4 \rtimes S_3$ on $X(\Delta)$.
    
    The group $V_4$ consists of the identity and three elements of order $2$. The $S_3$-action faithfully permutes these three non-identity elements. This structural property completely characterizes the symmetric group $S_4$, where $V_4$ sits as its normal Klein four-group. Therefore, $V_4 \rtimes S_3 \cong S_4$, which provides a faithful $S_4$-action on $X(\Delta)$ over $k$.
\end{proof}

\begin{example} \label{ex:S4_surfaces}
    By \cref{lem:S4_action_surface}, choosing appropriate $S_3$-orbits in $N$ as the rays of a complete fan $\Delta$ yields various toric surfaces admitting faithful $S_4$-actions over $k$. We illustrate three distinct examples in \cref{fig:S4_fans}.
\end{example}

    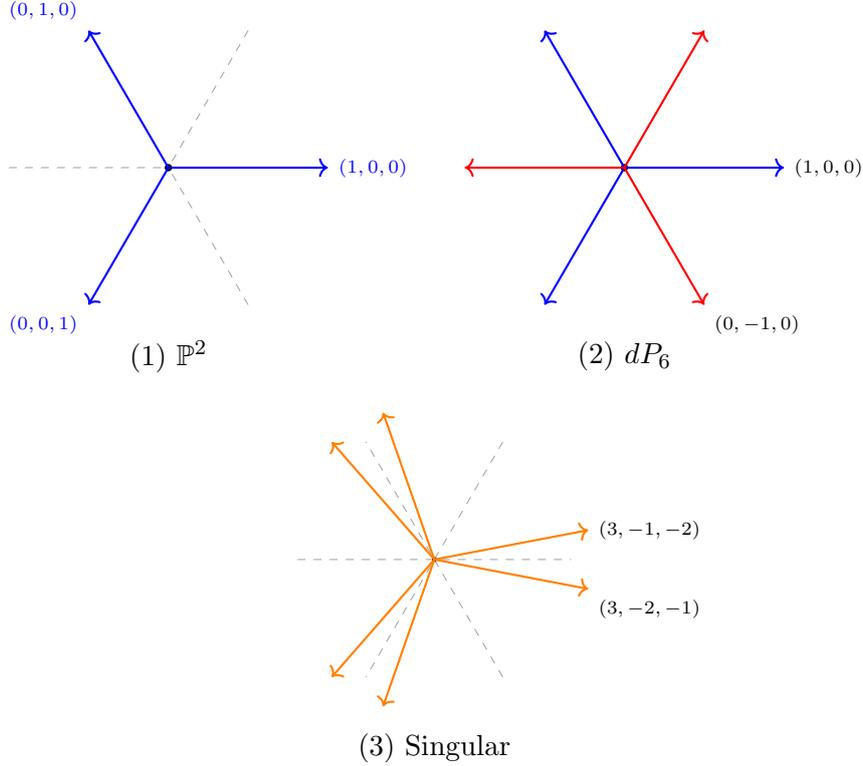
\begin{figure}[htbp]
        \centering
        \begin{tikzpicture}[
            xvec/.style={x={(1cm,0cm)}, y={(-0.5cm, 0.866cm)}, z={(-0.5cm, -0.866cm)}}
        ]
            \begin{scope}[xvec, shift={(-3,0)}, scale=0.7]
                \draw[gray!70, dashed] (-3,0,0) -- (3,0,0);
                \draw[gray!70, dashed] (0,-3,0) -- (0,3,0);
                \draw[gray!70, dashed] (0,0,-3) -- (0,0,3);
                
                \coordinate (v1) at (2,-1,-1);
                \coordinate (v2) at (-1,2,-1);
                \coordinate (v3) at (-1,-1,2);
                
                \fill (0,0,0) circle (2pt);
                \draw[->, thick, blue] (0,0,0) -- (v1) node[anchor=west] {\tiny $(1,0,0)$};
                \draw[->, thick, blue] (0,0,0) -- (v2) node[anchor=south east] {\tiny $(0,1,0)$};
                \draw[->, thick, blue] (0,0,0) -- (v3) node[anchor=north east] {\tiny $(0,0,1)$};
                
                \node[yshift=-2.5cm] at (0,0) {(1) $\PP^2$};
            \end{scope}

            \begin{scope}[xvec, shift={(3,0)}, scale=0.7]
                \draw[gray!70, dashed] (-3,0,0) -- (3,0,0);
                \draw[gray!70, dashed] (0,-3,0) -- (0,3,0);
                \draw[gray!70, dashed] (0,0,-3) -- (0,0,3);
                
                \coordinate (v1) at (2,-1,-1); \coordinate (v2) at (-1,2,-1); \coordinate (v3) at (-1,-1,2);
                
                \coordinate (e1) at (1,1,-2); \coordinate (e2) at (-2,1,1); \coordinate (e3) at (1,-2,1);
                
                \fill (0,0,0) circle (2pt);
                
                \foreach \p in {v1,v2,v3} { \draw[->, thick, blue] (0,0,0) -- (\p); }
                
                \foreach \p in {e1,e2,e3} { \draw[->, thick, red] (0,0,0) -- (\p); }
                
                \node[anchor=west] at (v1) {\tiny $(1,0,0)$};
                \node[anchor=north west] at (e3) {\tiny $(0,-1,0)$};
                
                \node[yshift=-2.5cm] at (0,0) {(2) $dP_6$};
            \end{scope}

            \begin{scope}[xvec, shift={(-2.5,-6)}, scale=0.45]
                \draw[gray!70, dashed] (-4,0,0) -- (4,0,0);
                \draw[gray!70, dashed] (0,-4,0) -- (0,4,0);
                \draw[gray!70, dashed] (0,0,-4) -- (0,0,4);
                
                \coordinate (u1) at (3,-1,-2); \coordinate (u2) at (3,-2,-1);
                \coordinate (u3) at (-1,3,-2); \coordinate (u4) at (-2,3,-1);
                \coordinate (u5) at (-1,-2,3); \coordinate (u6) at (-2,-1,3);
                
                \fill (0,0,0) circle (2pt);
                \foreach \p in {u1,u2,u3,u4,u5,u6} {
                    \draw[->, thick, orange] (0,0,0) -- (\p);
                }
                \node[anchor=west] at (u1) {\tiny $(3,-1,-2)$};
                \node[anchor=north west] at (u2) {\tiny $(3,-2,-1)$};
                
                \node[yshift=-2.5cm] at (0,0) {(3) Singular};
            \end{scope}
            
        \end{tikzpicture}
        \caption{Examples of fans in $N$ generating toric surfaces with $S_4$-actions.}
        \label{fig:S4_fans}
    \end{figure}

    \begin{enumerate}
        \item \textbf{The Projective Plane $\PP^2$}: The fan is generated by the $S_3$-orbit of $(1,0,0)$ in $N_2$, which consists of exactly $3$ rays. This forms $3$ smooth maximal cones, yielding the smooth projective plane $\PP^2$.
        
        \item \textbf{The del Pezzo Surface $dP_6$}: This surface is obtained by blowing up $\PP^2$ at the three torus-invariant points. Geometrically, the exceptional divisors correspond to rays obtained by summing adjacent rays of the $\PP^2$ fan. Thus, the fan consists of exactly \textit{two} distinct $S_3$-orbits of length $3$: the original orbit of $(1,0,0)$ and the newly added orbit of $(0,0,1)$.
        
        \item \textbf{A Singular Toric Surface}: Unlike the above $dP_6$, we can also generate a fan with $6$ rays from a \textit{single} $S_3$-orbit of length $6$, for example, the orbit of $(3,-1,-2)$ in $N_1$. The $2$-dimensional cone generated by adjacent rays, such as $(3,-1,-2)$ and $(3,-2,-1)$, is not smooth since their determinant in the basis of $N_1$ is $\pm 3 \neq \pm 1$. The resulting complete toric surface has $6$ quotient singularities, demonstrating that $S_4$ can act on infinitely many singular Fano toric surfaces.
    \end{enumerate}

\begin{lemma} \label{lemma:s4_nonsplit}
Let $L/k$ be an arbitrary quadratic Galois extension, and let $\langle \tau \rangle = \operatorname{Gal}(L/k) \cong \mathbb{Z}/2\mathbb{Z}$. Consider a lattice $N$ (either $N_1$ or $N_2$) equipped with an $S_3$-action as in \cref{lem:S4_action_surface}. Suppose $\tau$ acts on $N$ via the involution $\tau(v) = -v$. Since this involution corresponds to $-I \in \operatorname{GL}(N)$, it commutes with the $S_3$-action. 

Let $\Delta$ be a complete fan in $N_\mathbb{R}$ that is invariant under the action of $S_3 \times \operatorname{Gal}(L/k)$, and let $X_L(\Delta)$ be the split toric variety over $L$ defined by $\Delta$. If $X$ is the non-split $k$-form of $X_L(\Delta)$ associated with this Galois action, then $X$ admits a faithful $S_4$-action over $k$.
\end{lemma}

\begin{proof}
Since the Galois action $\tau(v) = -v$ commutes with the $S_3$-action on the fan, the $S_3$-action defined over $L$ is Galois-invariant. By the theory of Galois descent, this action uniquely descends to an $S_3$-action on the $k$-form $X$.

Furthermore, the toric variety $X$ contains a non-split algebraic torus $T$ over $k$. The Galois group acts on the points of the split torus $T_L \cong \mathbb{G}_{m,L}^2$ via the inversion $t \mapsto t^{-1}$. The elements of order $2$ in the torus (the $2$-torsion subgroup) are precisely the elements with coordinates $\pm 1$. Since the inversion map fixes $\pm 1$, the full $2$-torsion subgroup $V_4 \cong (\mathbb{Z}/2\mathbb{Z})^2$ is Galois-invariant and therefore descends to a subgroup of the $k$-rational points $T(k)$. 

Thus, the Klein four-group $V_4$ acts faithfully on $X$ over $k$. The descending $S_3$-action normalizes this $T$ and acts on $V_4$ by permuting its three non-identity elements. This group extension realizes the semi-direct product $V_4 \rtimes S_3$, which is isomorphic to the symmetric group $S_4$. We conclude that $X$ admits a faithful $S_4$-action over $k$.
\end{proof}

\begin{remark} \label{remark:imaginary_extension}
Recall that the condition \hyperref[cond:star]{\rm{($\star$)}} assumes that $-1$ is not a sum of two squares in the base field $k$. This immediately implies that $-1$ is not a square in $k$, meaning $\sqrt{-1} \notin k$. Consequently, for any field $k$ satisfying this condition, there always exists a quadratic Galois extension $L = k(\sqrt{-1})$ over $k$. As demonstrated in \cref{lemma:s4_nonsplit}, such an extension can be utilized to systematically construct non-split $k$-forms admitting a faithful $S_4$-action.
\end{remark}

\begin{example}\label{ex:Q22}
Consider the toric surface over $\mathbb{C}$ defined by the complete fan in $(N_2)_{\mathbb{R}}$ whose rays are generated by the $S_3 \times \operatorname{Gal}(\mathbb{C}/\mathbb{R})$-orbit of $(1,0,0)$. This surface is isomorphic to the degree $6$ del Pezzo surface $dP_6$, which is obtained by blowing up $\mathbb{P}^2_{\mathbb{C}}$ at three torus-invariant points. 

The $\mathbb{R}$-form $Q_{2,2}(0,1)$ corresponding to this $\operatorname{Gal}(\mathbb{C}/\mathbb{R})$-action admits a faithful action of the compact torus $S^1 \times S^1$. Because of this real structure, it is not isomorphic to the blow-up of $\mathbb{P}^2_{\mathbb{R}}$ at three real points (cf. \cite[\S 4]{yasinsky2022automorphisms}). 

Explicitly, it is given by the following subvariety:
\begin{multline*}
Q_{2,2}(0,1) = \{([x_0:x_1],[y_0:y_1],[z_0:z_1]) \in \mathbb{P}_\mathbb{R}^1 \times \mathbb{P}_\mathbb{R}^1 \times \mathbb{P}_\mathbb{R}^1 \mid \\
x_0y_0z_1 + x_0y_1z_0 + x_1y_0z_0 - x_1y_1z_1 = 0\}.
\end{multline*}
\end{example}

\begin{lemma}\label{lem:A4_action}
    Let $X$ be a complete simplicial toric surface over $\C$. Then $A_4 \le \Autzero(X)$ if and only if $X$ is isomorphic to one of the following:
    \begin{itemize}
        \item $\PP^2$,
        \item Hirzebruch surface $\mathcal{H}_a$ for some integer $a \in \Z$,
        \item $\PP(1,1,a)$ for some integer $a \in \Z$.
    \end{itemize}
\end{lemma}

\begin{proof}
Suppose that $A_4 \le \Autzero(X)$. Following the strategy in \cref{sec:dimension_4}, we consider the lifting of the $A_4$-action to the Cox ring of $X$. This induces a partition of the set of rays $\Delta(1) = \bigsqcup_{i=1}^s \Delta_i$ into orbits under the action of a central extension $\Gamma \le \prod_{i=1}^s \operatorname{GL}_{d_i}(\C)$, where $d_i = |\Delta_i|$. We may assume without loss of generality that $d_1 \ge d_2 \ge \dots \ge d_s$.

Since $A_4$ acts faithfully on $X$, there must be at least one block $\Delta_i$ with $d_i \ge 2$, as $A_4$ admits no non-trivial projective representation of dimension $1$. Considering that $\sum (d_i - 1) = \dim X = 2$, the possible partitions $(d_1, d_2, \dots)$ are limited to the following cases:

\noindent \textbullet\ \textit{Case $d_1 = 3$}: In this case, $X \cong \PP^2$ by \cite[Lemma 4.2]{esser2025symmetries}.
    
\noindent \textbullet\ \textit{Case $(d_1, d_2) = (2, 2)$} Also by \cite[Lemma 4.2]{esser2025symmetries}, $X$ is isomorphic to $\PP^1 \times \PP^1$, which is the Hirzebruch surface $\mathcal{H}_0$.
    
\noindent \textbullet\ \textit{Case $(d_1, d_2, d_3, \dots) = (2, 1, 1, \dots)$}: This case implies that $A_4$ acts non-trivially only on the first block $\Delta_1$. Following the same logic as in the proof of \cref{thm:main2}, the existence of a block of size $2$ combined with remaining blocks of size $1$ forces the variety to be a Hirzebruch surface $\mathcal{H}_a$ or a weighted projective counterpart $\PP(1,1,a)$.

Conversely, each of these surfaces admits a faithful $\Autzero(X)$-action containing $A_4$, which completes the proof.
\end{proof}

\begin{proposition} \label{prop:dim2_classification}
    Let $k$ be a field satisfying the condition \hyperref[cond:star]{\rm{($\star$)}}. Let $X$ be a complete simplicial toric surface over $k$. If the symmetric group $S_4$ acts faithfully on $X$, then $X$ is isomorphic to one of the following:
    \begin{itemize}
        \item a split toric surface constructed in \cref{lem:S4_action_surface},
        \item a non-split toric surface constructed in \cref{lemma:s4_nonsplit},
        \item the Weil restriction $\operatorname{Res}_{L/k}(\mathbb{P}^1_L)$, where $L/k$ is a quadratic Galois extension such that $-1$ is a sum of two squares in $L$.
    \end{itemize}
\end{proposition}

\begin{proof}
    Suppose $S_4 \le \Aut_k(X) \le \Aut(X_{\overline{k}})$. We consider the composition $\varphi \colon S_4 \to \Aut(X_{\overline{k}})/\Aut^0(X_{\overline{k}})$. The image $\operatorname{Im}(\varphi)$ must be isomorphic to the quotient of $S_4$ by its normal subgroup $K = S_4 \cap \Aut^0(X_{\overline{k}})$. The normal subgroups of $S_4$ are $\{1\}, V_4, A_4,$ and $S_4$. Therefore, the possible quotients are $S_4, S_3, \Z/2\Z,$ and $\{1\}$. We analyze each case for $\operatorname{Im}(\varphi)$:

    \noindent \textbullet\ \textit{Case $\operatorname{Im}(\varphi) \cong S_4$}: According to \cite[Lemma 4.5]{esser2025symmetries}, the discrete part $\Aut(X)/\Aut^0(X)$ for a toric surface cannot contain a subgroup isomorphic to $S_4$. Thus, this case cannot occur.

    \noindent \textbullet\ \textit{Case $\operatorname{Im}(\varphi) \cong S_3$}: The kernel is $K = V_4 \le \operatorname{Aut}^0(X_{\overline{k}})$. As argued in the proof of \cite[Lemma 4.5]{esser2025symmetries}, this requires $S_3$ to act faithfully on the $2$-dimensional vector space $N_\mathbb{R}$. Up to isomorphism, the faithful $2$-dimensional representation of $S_3$ over $\mathbb{Z}$ is the standard representation on $N_1 = \{(x,y,z) \in \mathbb{Z}^3 \mid x+y+z=0\}$ or $N_2 = \mathbb{Z}^3/\mathbb{Z}(1,1,1)$ given by permuting the coordinates (\cite[Chapter IX, Section 14]{newman1972integral}). This forces the combinatorial structure of $\Delta$ to be exactly the $S_3$-invariant fan described in \cref{lem:S4_action_surface}.

    If $X$ is a split toric surface over $k$, the fan $\Delta$ is fully defined over $k$. Since the variety is split, the $S_4$-action on $X_{\overline{k}}$ naturally descends to $k$. Thus, $X$ is isomorphic to one of the split toric surfaces constructed in \cref{lem:S4_action_surface}.

    Suppose $X$ is a non-split toric surface over $k$, and let $L$ be its splitting field. Over $L$, $X_L$ is a split toric surface admitting an $S_4$-action, meaning its fan is one of those described above. For this $S_4$-action to descend to $k$, the Galois action of $\operatorname{Gal}(L/k)$ on the lattice $N$ must commute with the $S_3$-action and preserve the fan $\Delta$. Since the standard representation of $S_3$ on $N \otimes \mathbb{Q}$ is absolutely irreducible, Schur's Lemma implies that the centralizer of $S_3$ in $\operatorname{GL}(N)$ consists only of the scalar matrices $\pm I$. Because $X$ is non-split, the Galois action on $N$ cannot be trivial. Therefore, the Galois group must act via the involution $\tau(v) = -v$, which corresponds to a quadratic extension $L/k$. This implies that $X$ is precisely the non-split toric surface constructed in \cref{lemma:s4_nonsplit}.

    \noindent \textbullet\ \textit{Case $\operatorname{Im}(\varphi) \cong \mathbb{Z}/2\mathbb{Z}$ or $\{1\}$}: In these cases, the kernel $K$ contains the alternating group $A_4$. This means $A_4$ acts purely through the continuous part of the automorphism group, i.e., $A_4 \le \operatorname{Aut}^0(X_{\overline{k}})$. 

    If $X$ is a split toric surface, we apply a lifting argument similar to that in the proof of \cref{prop:dim3}. The $A_4$-action on $X$ lifts to an action of a central extension $\Gamma$ on the Cox ring over $k$, yielding an embedding $\Gamma \hookrightarrow \prod_{i=1}^s \operatorname{GL}_{d_i}(k)$. Under the condition \hyperref[cond:star]{\rm{($\star$)}}, $A_4$ cannot be embedded into $\operatorname{PGL}_2(k)$, which rules out the possibility of a maximum block size $d_1 = 2$. Consequently, we must have $d_1 = 3$, which uniquely identifies the variety as $X \cong \mathbb{P}^2_k$.

    Now consider the case where $X$ is a non-split toric surface. For any toric surface in \cref{lem:A4_action} other than $\mathbb{P}^1 \times \mathbb{P}^1$, the symmetry of the ray configuration only allows for permutations within the same blocks of rays. Up to the torus action, such $k$-forms are necessarily isomorphic to the split toric surfaces already considered. Thus, the only non-split candidate requiring further investigation is the Weil restriction $X = \operatorname{Res}_{L/k}(\mathbb{P}^1_L)$. The automorphism group of this variety is given by:
    \[
        \operatorname{Aut}_k(X) \cong \operatorname{PGL}_2(L) \rtimes \operatorname{Gal}(L/k).
    \]
    For $S_4$ to be a subgroup of $\operatorname{Aut}_k(X)$, the alternating group $A_4$ must admit a faithful representation in $\operatorname{PGL}_2(L)$. This group-theoretic requirement is equivalent to the condition that $-1$ is a sum of two squares in $L$. Conversely, if this condition is satisfied, the embedding $A_4 \le \operatorname{PGL}_2(L)$ exists, and since $\operatorname{PGL}_2(L)$ contains $S_4$ whenever it contains $A_4$ (via the addition of an appropriate involution), the symmetric group $S_4$ acts faithfully on $X$ over $k$.
    
    This exhausts all cases and concludes the proof.
\end{proof}

\begin{proposition} \label{prop:smooth_S4_MMP}
    Let $X$ be a smooth complete toric surface over $k$ admitting a faithful $S_4$-action. Then running the $S_4$-equivariant Minimal Model Program (MMP) on $X$ terminates with one of the following surfaces:
    \begin{itemize}
        \item the split projective plane $\PP^2_k$,
        \item the split smooth del Pezzo surface $dP_6$,
        \item the Weil restriction $\operatorname{Res}_{L/k}(\PP^1_L)$ (as in \cref{prop:dim2_classification}), or
        \item the non-split toric surface $Q_{2,2}(0,1)_k$ (as in \cref{ex:Q22}).
    \end{itemize}
    
    In particular, the geometrical structure of $X$ is characterized as follows:
    \begin{enumerate}
        \item Any toric blow-up of $\operatorname{Res}_{L/k}(\PP^1_L)$ fails to admit a faithful $S_4$-action. Thus, if the MMP terminates at this Weil restriction, $X$ must be exactly $\operatorname{Res}_{L/k}(\PP^1_L)$ itself.
        \item If $X$ is a split toric surface, it is obtained from $\PP^2_k$ by a finite sequence of $S_4$-equivariant blow-ups at $S_3$-orbits of torus-invariant points.
        \item If $X$ is a non-split toric surface descending from $Q_{2,2}(0,1)$, it is obtained by a finite sequence of $S_4$-equivariant blow-ups from $Q_{2,2}(0,1)$. Over the splitting field $L$, this process corresponds to a tower of blow-ups at $S_3 \times \operatorname{Gal}(L/k)$-orbits of torus-invariant points. Unlike the Weil restriction, this geometrical process generates an infinite family of smooth non-split toric surfaces admitting a faithful $S_4$-action.
    \end{enumerate}
\end{proposition}

\begin{proof}
    To execute the $S_4$-equivariant Minimal Model Program (MMP) over $k$, we pass to the splitting field $L$ (where $L=k$ if $X$ is split, and $[L:k]=2$ if $X$ is non-split). Over $L$, $X_L$ is a smooth split toric surface. The $S_4$-action on $X$ induces a faithful action of $S_3$ on the fan $\Delta$, and if $X$ is non-split, the Galois group $\operatorname{Gal}(L/k) = \langle \tau \rangle$ also acts on $\Delta$. We can therefore run the equivariant MMP on $X_L$ with respect to this symmetry group.

    By the standard MMP for toric surfaces, any smooth complete toric surface is obtained by successive blow-ups from either $\PP^2$ or a Hirzebruch surface $\mathcal{H}_r$. Thus, if $|\Delta(1)| \ge 5$, $X_L$ must contain at least one $(-1)$-curve. We track the orbits of these $(-1)$-curves under the relevant symmetry group.

    \noindent \textbf{Case 1: $X$ is a split toric surface ($L=k$).} \\
    The symmetric group acting on $\Delta$ is exactly $S_3$. We must rule out the Hirzebruch surfaces $\mathcal{H}_r$ as possible terminal outputs. The fan of $\mathcal{H}_r$ consists of exactly $4$ rays. Since $4$ is not divisible by $3$, any $S_3$-action on $\Delta(1)$ must fix at least one ray. However, the only vector fixed by the faithful $2$-dimensional representation of $S_3$ is the origin, making such an action impossible. Thus, $\mathcal{H}_r$ never appears.
    
    If $|\Delta(1)| \ge 5$, we consider the configuration of the $S_3$-orbit of a $(-1)$-curve:
    \begin{itemize}
        \item \textit{Disjoint orbit:} If these $(-1)$-curves are mutually non-adjacent, we can equivariantly contract this entire orbit over $k$ to obtain a new smooth toric surface $X'$ still admitting the $S_4$-action.
        \item \textit{Adjacent orbit:} If there are adjacent $(-1)$-curves in the orbit, let their corresponding rays be $v_1$ and $v_2$. In a smooth fan, the properties of $(-1)$-curves force their neighboring rays $v_0$ and $v_3$ (ordered sequentially) to satisfy:
        \[
            v_0 + v_2 = v_1 \quad \text{and} \quad v_1 + v_3 = v_2.
        \]
        Adding these equations yields $v_0 + v_3 = 0$, meaning $v_0 = -v_3$. Because the $S_3$-action is faithful and preserves the fan, this configuration of opposing rays forces the fan to have exactly $6$ rays forming a single $S_3$-orbit. This uniquely identifies the fan as that of the smooth del Pezzo surface $dP_6$. Since further contraction introduces singularities, $dP_6$ is a terminal output.
    \end{itemize}
    Therefore, the split $S_4$-MMP over $k$ must terminate at either $\PP^2_k$ (if the ray count drops to $3$) or $dP_{6,k}$ (if an adjacent orbit is reached). Any split surface is thus obtained by a finite sequence of $S_4$-equivariant blow-ups from $\PP^2_k$.

    \noindent \textbf{Case 2: $X$ is a non-split toric surface ($[L:k]=2$).} \\
    By \cref{prop:dim2_classification}, $X$ is either the Weil restriction $\operatorname{Res}_{L/k}(\PP^1_L)$ or a form associated with the involution $\tau(v) = -v$.

    \textit{Subcase 2a: $X = \operatorname{Res}_{L/k}(\PP^1_L)$.} Over $L$, $X_L \cong \PP^1_L \times \PP^1_L$, and its fan consists of $4$ rays with no $(-1)$-curves. It is already a minimal model. Furthermore, any toric blow-up of this surface destroys the delicate $S_4 \le \operatorname{PGL}_2(L) \rtimes \operatorname{Gal}(L/k)$ representation. Thus, the Weil restriction stands alone as an isolated terminal model that generates no further $S_4$-surfaces via blow-ups.

    \textit{Subcase 2b: $\tau(v) = -v$.} The relevant symmetry group on $\Delta$ is $S_3 \times \operatorname{Gal}(L/k)$. The $S_3 \times \operatorname{Gal}(L/k)$-orbit of a $(-1)$-curve $v$ is the union $S_3v \cup S_3(-v)$. If we encounter an orbit of disjoint $(-1)$-curves, we can equivariantly blow them down over $k$. If the orbit contains adjacent curves, the exact same relation $v_0 + v_3 = 0$ arises, again forcing the $dP_6$ configuration over $L$. 
    
    Crucially, in the $dP_6$ fan, the $S_3$-orbits of $(-1)$-curves have size $3$ (the two sets of disjoint lines). The Galois involution $\tau(v) = -v$ precisely swaps these two disjoint orbits. Consequently, the $S_3 \times \operatorname{Gal}(L/k)$-orbit consists of all $6$ rays. It is impossible to blow down all $6$ rays simultaneously while maintaining a valid surface. Therefore, the MMP must terminate here. Over $k$, this terminal non-split $dP_6$ is exactly $Q_{2,2}(0,1)$.

    Reversing this process, one can systematically blow up $Q_{2,2}(0,1)$ at $S_3 \times \operatorname{Gal}(L/k)$-orbits of torus-invariant points. Because the Galois action ensures the blow-ups descend to $k$, this constructs an infinite tower of smooth non-split toric surfaces admitting a faithful $S_4$-action, completely characterizing the non-split family.
\end{proof}

\begin{remark} \label{rem:two_S4_actions}
    It is worth noting that the smooth del Pezzo surface $dP_{6,k}$ admits two non-conjugate faithful $S_4$-actions over $k$, which exhibit completely different behaviors with respect to the $S_4$-MMP.
    
    The fan of $dP_6$ consists of $6$ rays forming a regular hexagon. Its combinatorial automorphism group is the dihedral group $D_{12}$ of order $12$. The group $D_{12}$ contains two non-conjugate subgroups isomorphic to $S_3$, which precisely correspond to the two different integral representations (lattices) discussed above:
    \begin{enumerate}
        \item In the lattice $N_2 = \Z^3/\Z(1,1,1)$, the fan of $dP_6$ is generated by two distinct $S_3$-orbits of length $3$ (e.g., the orbit of $[1,0,0]$ and the orbit of $[-1,0,0]$). Under this $S_4$-action, the $(-1)$-curves form two separate orbits. We can choose one of these orbits and contract it equivariantly. This allows the $S_4$-MMP to continue, ultimately yielding $\PP^2_k$.
        
        \item In the lattice $N_1 = \{(x,y,z) \in \Z^3 \mid x+y+z=0\}$, the fan of $dP_6$ is generated by a single $S_3$-orbit of length $6$ (e.g., the orbit of $(1,-1,0)$). Under this $S_4$-action, all six $(-1)$-curves form a single orbit. As shown in \cref{prop:smooth_S4_MMP}, any equivariant contraction is impossible because adjacent $(-1)$-curves cannot be contracted smoothly. Thus, under this specific $S_4$-action, $dP_{6,k}$ acts as a terminal output of the $S_4$-MMP.
    \end{enumerate}
\end{remark}

\bibliographystyle{amsalpha}
\bibliography{Ref}

\end{document}